\documentclass[a4paper,english]{article}
\usepackage[latin9]{inputenc}
\usepackage{geometry}
\usepackage{float}
\usepackage{amsthm}
\usepackage{amsmath}
\usepackage{amstext}
\usepackage{amssymb}
\usepackage{color}
\usepackage{graphicx}
\usepackage{nicefrac}
\usepackage{fancyhdr}
\usepackage{lastpage}
\usepackage{babel}
\usepackage{amsopn}
\usepackage{slashbox}
\usepackage{tikz}
\usepackage{rotfig}
\usepackage{graphicx}
\usepackage{xcolor}
\usepackage{rotating}
\usetikzlibrary{arrows,chains,matrix,positioning,scopes,shapes,arrows}
\usepackage{blkarray}
\usepackage{multirow}
 \usepackage{booktabs}
\usepackage{algorithm}
\usepackage{hyperref}
\usepackage{algorithmic}
\usepackage{listings} 
\usepackage{tikz}
\usepackage{ifthen}
\usepackage{url}

\usepackage[font={small,it}]{caption}

\usepackage{amsthm}
\usepackage{thmtools}

\newcommand{\smb}{\left[\begin{smallmatrix}}
\newcommand{\sme}{\end{smallmatrix}\right]}

\newtheorem{theorem}{Theorem}
\newtheorem{remark}[theorem]{Remark}

\newtheorem{example}[theorem]{Example}
 
\newtheorem{lemma}[theorem]{Lemma}

\DeclareMathOperator{\rank}{rank}
\DeclareMathOperator{\h}{\mathcal{H}}
\DeclareMathOperator{\diam}{diam}
\DeclareMathOperator{\diag}{diag}

\DeclareMathOperator{\gap}{\mathsf{gap}}
\DeclareMathOperator{\dist}{dist}

\DeclareMathOperator{\myspan}{span}
\DeclareMathOperator{\sign}{sign}
\DeclareMathOperator{\off}{off}
\DeclareMathOperator{\e}{e}
\DeclareMathOperator{\chol}{chol}
\DeclareMathOperator{\id}{id}
\DeclareMathOperator{\trace}{trace}
\DeclareMathOperator{\SP}{SP}
\DeclareMathOperator{\sn}{sn}

\newcommand{\R}{{\mathbb R}}

\newcommand{\Matlab}{{\sc Matlab}}

\newcommand{\calO}{{\mathcal O}}
\newcommand{\calU}{{\mathcal U}}

\newcommand*\circled[1]{\tikz[baseline=(char.base)]{
            \node[shape=circle,draw,inner sep=0.5pt] (char) {#1};}}
\newcommand*\circledt[1]{\tikz[baseline=(char.base)]{
            \node[shape=circle,draw,inner sep=0.05pt] (char) {#1};}}
\newcommand*\circledr[1]{\tikz[baseline=(char.base)]{
            \node[shape=rectangle,draw,inner sep=0.5pt] (char) {#1};}}       
\tikzstyle{block} = [rectangle, draw, fill=blue!20, 
    text width=5em, text centered, rounded corners, minimum height=4em]
            
\newenvironment{lbmatrix}[1]
  {\left[\array{@{}*{#1}{c}@{}}}
  {\endarray\right]}

\title{Fast computation of spectral projectors \\ of banded matrices}
\author{Daniel Kressner\thanks{MATHICSE-ANCHP, \'{E}cole Polytechnique F\'{e}d\'{e}rale de Lausanne, Station 8, 1015 Lausanne, Switzerland. E-mail: daniel.kressner@epfl.ch.} \and  Ana 
\v{S}u\v{s}njara\thanks{MATHICSE-ANCHP, \'{E}cole Polytechnique F\'{e}d\'{e}rale de Lausanne, Station 8, 1015 Lausanne, Switzerland. E-mail: ana.susnjara@epfl.ch. The work of Ana \v{S}u\v{s}njara
has been supported by the SNSF research project \emph{Low-rank updates of matrix functions and fast eigenvalue solvers.} }}

\begin{document}

\date{}

\maketitle

\begin{abstract}
We consider the approximate computation of spectral projectors for symmetric banded matrices. While this problem has received considerable attention, 
especially in the context of linear scaling electronic structure methods, the presence of small relative spectral gaps 
challenges existing methods based on approximate sparsity. In this work, we show how a data-sparse 
approximation based on hierarchical matrices can be used to overcome this problem. We prove a priori bounds on the 
approximation error and propose a fast algorithm based on the QDWH algorithm, along the works by Nakatsukasa et al. 
Numerical experiments demonstrate that the performance of our algorithm is robust with respect to the spectral gap. A
preliminary \Matlab{} implementation becomes faster than {\tt eig} already for matrix sizes of a few thousand.
\end{abstract}

\section{Introduction}

Given a symmetric banded matrix $A\in \mathbb{R}^{n\times n}$ with eigenvalues
\begin{equation*}
\lambda_{1} \leq  \cdots \leq \lambda_{\nu} < \mu < \lambda_{\nu+1} \leq \cdots \leq \lambda_{n},
\end{equation*}
we consider the computation of the  spectral projector $\Pi_{< \mu}(A)$ associated with the eigenvalues $\lambda_{1}, \ldots ,\lambda_{\nu}$. We specifically target the 
situation where both $n$ and $\nu$ are large, say $n = 100\,000$ and $\nu = 50\,000$, which makes approaches based on computing eigenvectors computationally expensive. For a tridiagonal matrix, 
the MRRR algorithm requires $\mathcal{O}(\nu n)$ operations and memory~\cite{DhillParlVoemel2006} 
to compute the $\nu$ eigenvectors needed to define $\Pi_{< \mu}(A)$.

There are a number of applications giving rise to the problem under consideration. First and foremost, 
this task is at the heart of linear scaling methods for the calculation of the electronic structure of molecules with a large number of atoms. For insulators at zero 
temperature, the density matrix is the spectral projector associated with the eigenvalues of the Hamiltonian below the so called 
HOMO-LUMO gap; see~\cite{Goedecker1999} for an overview. The Hamiltonian is usually symmetric and, depending on the discretization and the structure of the molecule, 
it can be (approximately) banded. A number of existing linear scaling methods use that this sometimes implies that the spectral projector may also admit a good approximation by a 
banded matrix; see~\cite{BenBoiRaz2013} for a recent survey and a mathematical justification. For this approach to work well, the HOMO-LUMO gap should not become too small. For metallic systems, 
this gap actually converges to zero, which makes it impossible to apply an approach based on approximate bandedness or, more generally, sparsity.

Another potential important application for banded matrices arises in dense symmetric eigenvalue solvers. The eigenvalues and eigenvectors of a symmetric dense 
matrix $A$ are usually computed by first reducing $A$ to tridiagonal form and then applying either divide-and-conquer method or MRRR; see, e.g.~\cite{AuckBungHuckLang2011,DemmMarqParlVoem2008}
for recent examples. It is by no means trivial to implement the reduction to tridiagonal form efficiently so that it performs well on a modern computing architecture with a memory hierarchy. Most 
existing approaches~\cite{AuckBlumBungHuck2011,BienIgualKressPet2011,HaidLtaiDong2011,HaidSolcGates2013,SoloBallDemmHoef016}, with the notable exception of~\cite{PetsPeisBien2013},
 are based on successive band reduction~\cite{BiscLangSun2000}. In this context, it would be preferable to design an eigenvalue solver 
that works directly with banded matrices, bypassing the need for tridiagonal reduction. While we are not aware of any such extension of MRRR, this possibility has been explored
several times for the divide-and-conquer method, e.g., in~\cite{Arbenz1992,HaidLtaiDong2012}. The variants proposed so far seem to suffer either from numerical instabilities
or from a complexity that grows significantly with the bandwidth. The method proposed in this paper can be used to directly compute the spectral 
projector of a banded matrix, which in turn could potentially be used as a basis for a fast spectral divide and conquer algorithm in the spirit of Nakatsukasa and Higham~\cite{NakaHigh2013}. 

To deal with small spectral gaps, one needs to go beyond sparsity. It turns out that 
hierarchical matrices~\cite{Hackbusch1999}, also called $\h$--matrices, are much better suited in such a setting. 
Intuitively, this can be well explained by considering the approximation of the Heaviside function $\Pi_{< \mu}(x)$ on the eigenvalues of $A$. 
While a polynomial approximation of $\Pi_{< \mu}$ corresponds to a sparse approximation of $\Pi_{< \mu}(A)$~\cite{BenBoiRaz2013}, a rational approximation corresponds to 
an approximation of $\Pi_{< \mu}(A)$ that features hierarchical low-rank structure. It is well known, see, e.g.,~\cite{PetrPop1987}, that a rational 
approximation is more powerful in dealing with nearby singularities, such as $x = \mu$ for $\Pi_{< \mu}(x)$. 

There are a number of existing approaches to use hierarchical low-rank structures for the fast computation of matrix functions, 
including spectral projectors. Beylkin, Coult, and Mohlen\-kamp~\cite{BeylCoultMohl1999} proposed a combination of the Newton--Schulz iteration 
with the HODLR format, a subset of $\h$--matrices, to compute spectral projectors for banded matrices. However,
the algorithm does not fully exploit the potential of low-rank formats; it converts a full matrix to the HODLR format in each iteration. In the context of Riccati and Lyapunov matrix equations, 
the computation of the closely related sign function of an $\h$--matrix has been discussed in~\cite{GrasHackKhor2003,BaurBenner2006}. 
The work in~\cite{GavrHackKhor2002,GavrHackKhor2004,GrasHackKhor2003} involves the $\h$--matrix approximation of resolvents, which is then used to compute the matrix exponential and related matrix functions.

Other hierarchical matrix techniques for eigenvalue problems include slicing-the-spectrum, which uses LDL decompositions to 
compute eigenvalues in a specified interval for symmetric HODLR and HSS matrices~\cite{BennerMach2012} as well as $\h^{2}$--matrices~\cite{BennBoermMach2013}. Approximate 
$\h$--matrix inverses can be used as preconditioners in iterative eigenvalue solvers; see~\cite{Lintner2002,Mach2012} for examples. Recently, Vogel et al.~\cite{VogelXiaCauBal2016} 
have developed a fast divide-and-conquer method for computing all eigenvalues and eigenvectors in the HSS format. However, as the matrix of eigenvectors is represented in a factored form, it would be a 
nontrivial and possibly expensive detour to compute spectral projectors via this approach.

In this paper we propose a new method based on a variant~\cite{NakaHigh2013} of the QR-based dynamically weighted Halley algorithm (QDWH) 
for computing a polar decomposition~\cite{NakaBaiGygi2010}. Our method exploits the fact that the iterates of QDWH applied to a banded matrix 
can be well approximated in the HODLR format. In fact, we show that the memory needed for storing the approximate spectral projector depends only logarithmically 
on the spectral gap, a major improvement over approximate sparsity. The implementation of QDWH requires some care, in particular, 
concerning the representation of the first iterate. One major contribution of this work is to show how this can be done efficiently.

The remainder of the paper is organized as follows. In Section~\ref{qdwh_algorithm}, we review the QDWH algorithm for computing a spectral projector $\Pi_{< \mu}(A)$. Section~\ref{hmatrices} recalls well-known
facts about the HODLR format and the corresponding formatted arithmetics. Based on the best rational approximation to the sign function, we derive new a priori bounds on the singular values 
for off-diagonal blocks of $\Pi_{< \mu}(A)$, from which we deduce bounds on the memory required to store $\Pi_{< \mu}(A)$ approximately in the HODLR format. Section~\ref{qr-based-iteration} discusses 
the efficient realization of the QR decomposition required in the first iterate of the QDWH algorithm. Section~\ref{overall_algorithm} summarizes our newly proposed QDWH algorithm in the HODLR format and 
provides implementation details. Finally, numerical experiments both for tridiagonal and banded matrices are shown in Section~\ref{experiments}.

\section{Computation of spectral projectors via QDWH}
\label{qdwh_algorithm}

In the following, we assume $\mu = 0$ without loss or generality, and thus consider the computation of the spectral projector $\Pi_{< 0}(A)$ associated with the negative eigenvalues of a 
symmetric nonsingular matrix $A \in \R^{n\times n}$. Following~\cite{NakaHigh2013}, our approach is based on a well-known connection to the polar decomposition.

The polar decomposition~\cite[Chapter~9]{GolVanL2013} of $A$ takes the form $A = UH$ for an orthogonal matrix $U$ and a symmetric positive definite matrix $H$. Let 
$A  = V\Lambda V^{*}$ be a spectral decomposition of $A$ such that $\Lambda = \operatorname{diag}(\Lambda_{-}, \Lambda_{+})$, where $\Lambda_{-}$ and $\Lambda_{+}$ 
are diagonal matrices containing the $\nu$ negative and the $n-\nu$ positive eigenvalues of $A$, respectively. Then
\begin{align}
 A &= V \operatorname{diag}(\Lambda_{-}, \Lambda_{+}) V^{*} \notag\\
   &= \underbrace{V \operatorname{diag}(-I_{\nu}, I_{n-\nu}) V^{*}}_{=:U} \cdot \underbrace{V\operatorname{diag}(\vert \Lambda_{-}\vert, \vert \Lambda_{+}\vert) V^{*}}_{=:H} \notag 
\end{align}
gives the polar decomposition of $A$.
In particular, this shows that the matrix sign function $\sign(A)$ coincides with the orthogonal factor $U$ from the polar decomposition. 
More importantly, $\Pi_{< 0}(A) = \frac{1}{2}(I - U)$.

\subsection{QDWH algorithm}
\label{qdwh_algorithm1}
The QDWH algorithm~\cite{NakaBaiGygi2010} computes the polar factor $U$ of $A$ as the limit of the sequence $X_k$ defined by 
\begin{align}
\label{eq:reccur}
X_0 &= A/\alpha\text{,}\notag \\
X_{k+1} &= X_k(a_k I + b_kX_k^{*}X_k)(I + c_kX_k^{*}X_k)^{-1}\text{.}
\end{align}
The parameter $\alpha > 0$ is an estimate of $\Vert A \Vert_2$. The parameters $a_k, b_k, c_k$ are computed via the relations 
\begin{equation}
\label{eq:qdwh_parameters_abc}
 a_k = h(l_k),\quad b_k = (a_k-1)^2/4, \quad c_k = a_k + b_k - 1.
\end{equation}
Representing a lower bound for the smallest singular value of $X_k$, the parameter $l_k$  is determined by the recurrence 
\[
l_k = l_{k-1}(a_{k-1} + b_{k-1}l^{2}_{k-1})/(1 + c_{k-1}l^{2}_{k-1}), \quad k \geq 1,
\]
where $l_{0}$ is a lower bound for $\sigma_{\min}(X_0)$.
The function $h$ is given by
\begin{equation*}
h(l) = \sqrt{1+ \gamma} + \frac{1}{2}\sqrt{8 - 4\gamma + \frac{8(2 - l^2)}{l^2\sqrt{1 + \gamma}}},\quad \gamma = \sqrt[3]{\frac{4(1 - l^2)}{l^4}}.
\end{equation*}
The efficient estimation of $\alpha$ and $l_{0}$, required to start the recurrence, will be discussed in Section~\ref{overall_algorithm}.  

The QDWH algorithm is cubically convergent and it has been shown in~\cite{NakaBaiGygi2010} that at most $k = 6$ iterations are needed to obtain convergence within tolerance $10^{-16}$, i.e. 
$\Vert X_6 - U \Vert_2 < 10^{-16}$ for every matrix $A$ with $\kappa(A) \leq 10^{16}$. 

The recurrence~\eqref{eq:reccur} has the equivalent form
\begin{subequations}
\label{eq:qdwh_qr_reccur}
\begin{align}
 X_0 &= A/\alpha \label{eq:qdwh_it_0},\\
X_{k+1} &= \frac{b_k}{c_k}X_k + \frac{1}{\sqrt{c_k}}\left(a_k - \frac{b_k}{c_k}\right)Q_1Q_2^{*}\label{eq:qdwh_qr} \text{,}
\end{align}
\end{subequations}
with the QR decomposition
\begin{equation} 
\label{eq:qdwh_qr_decomposition}
 \begin{bmatrix}
  \sqrt{c_k}X_k \\
  I
 \end{bmatrix}  =  
 \begin{bmatrix}
  Q_1 \\
  Q_2
 \end{bmatrix}R\text{.} 
\end{equation}
Throughout the paper, we refer to~\eqref{eq:qdwh_qr_reccur} as a \emph{QR-based iteration}. On the other hand, as observed in~\cite{NakaHigh2013}, the recurrence~\eqref{eq:reccur} can also be rewritten in terms of the \emph{Cholesky-based iteration}
\begin{subequations}
\label{eq:qdwh_chol_reccur}
\begin{align}
 Z_k &= I + c_k X_k^{*}X_k, \hskip 5pt W_k = \chol(Z_k), \label{eq:qdwh_chol_1}\\ 
 X_{k+1} &= \frac{b_k}{c_k}X_k + \left( a_k - \frac{b_k}{c_k}\right)(X_kW_k^{-1})W_k^{-*}\label{eq:qdwh_chol2},
\end{align} 
\end{subequations}
where $\chol(Z_k)$ denotes the Cholesky factor of $Z_k$. 

Following~\cite{NakaBaiGygi2010}, either variant of the QDWH 
algorithm is terminated when $l_k$ is sufficiently close to $1$, that is, $\vert 1 - l_k\vert \leq \delta$ for some stopping tolerance $\delta$, say $\delta = 10^{-15}$.

We mention that a higher--order variant of QDWH, called Zolo-pd, has recently been proposed by Freund and Nakatsukasa~\cite{NakaFreund2015}. This method approximates the polar 
decomposition in at most two iterations but requires more arithmetic per iteration.

\subsection{Switching between QR-based and Cholesky-based iterations}

Due to its lower operation count, it can be expected that one Cholesky-based iteration~\eqref{eq:qdwh_chol_reccur} is faster than one QR-based iteration~\eqref{eq:qdwh_qr_reccur}. However, when $Z_k$ is ill-conditioned, 
which is signaled by a large value of $c_k$,
the numerical stability of~\eqref{eq:qdwh_chol_reccur} can be jeopardized. To avoid this, it is proposed in~\cite{NakaHigh2013} to switch from~\eqref{eq:qdwh_qr_reccur} to~\eqref{eq:qdwh_chol_reccur} as soon as $c_k \leq 100$. 
Since $c_k$ converges monotonically from above to $3$, this implies that this hybrid approach will first perform a few QR-based iterations and then switch for good to Cholesky-based iterations. 
In fact, numerical experiments presented in~\cite{NakaHigh2013} indicate that at most two QR-based iterations are performed.

For reasons explained in Remark~\ref{remark:qr} below, we prefer to perform only one QR-based iteration and then switch to Cholesky-based iterations. To explore the impact of this choice on numerical accuracy, 
we perform a comparison of the QDWH algorithm proposed in~\cite{NakaHigh2013} with a variant of QDWH that performs only one QR-based iteration. We consider the following error measures: 
\begin{align} \label{eq:error_measures}
\begin{split}
  e^{Q}_{\id} &:= \Vert U^2 - I\Vert_2,\\ 
 e^{Q}_{\trace}&:= \vert \trace(U) - \trace(\sign(A))\vert,\\
 e^{Q}_{\SP} &:= \Big\Vert \frac{1}{2}(I-U) - \Pi_{<0}(A)\Big\Vert_2, 
 \end{split}
\end{align}
where $U$ denotes the output of the QDWH algorithm, and $\Pi_{<0}(A)$ the spectral projector returned by the \Matlab{} function \texttt{eig}.

\begin{example} \rm 
Let $A\in \mathbb{R}^{2000\times 2000}$ be a symmetric tridiagonal matrix constructed as described in Section~\ref{sec:construct_test_matrices}, such that 
half of the spectrum of $A$ is contained in $[-1, \hskip 3pt -\gap]$ and the other half in $[\gap, \hskip 3pt 1]$, for 
$\gap \in \left\lbrace  10^{-1}, 10^{-5}, 10^{-10}, 10^{-15}\right\rbrace$. 
\begin{table}[h!]
\centering
{\renewcommand{\arraystretch}{1.2}
\begin{tabular}{c|c|c|c|c|c}
Algorithm~\cite{NakaHigh2013} &gap &$10^{-1}$ &$10^{-5}$ &$10^{-10}$ &$10^{-15}$\\
 \hline
 \hline
\multirow{3}{6.5em}{one QR-based iteration~\eqref{eq:qdwh_qr_reccur}}&$e^{Q}_{\trace}$ &$5.55\cdot 10^{-17}$  &$7.22\cdot 10^{-16}$ &$2.22\cdot 10^{-16}$ &$1.11\cdot 10^{-16}$\\  
 \cline{2-6}
&$e^{Q}_{\id}$  &$1.15\cdot 10^{-15}$ &$2.41\cdot 10^{-15}$ &$1.84\cdot 10^{-15}$ &$1.82\cdot 10^{-15}$ \\
\cline{2-6}
&$e^{Q}_{\SP}$  &$1.87\cdot 10^{-14}$ &$4.35\cdot 10^{-12}$ &$1.88\cdot 10^{-6}$ & $1.91\cdot 10^{-2}$\\
\hline
\hline
\multirow{4}{6.5em}{several QR-based iterations~\eqref{eq:qdwh_qr_reccur}}&$e^{Q}_{\trace}$ &$5.55\cdot 10^{-17}$  &$1.22\cdot 10^{-15}$ &$1.53\cdot 10^{-16}$ &$6.25\cdot 10^{-16}$\\  
 \cline{2-6}
&$e^{Q}_{\id}$ &$1.15\cdot 10^{-15}$  &$2.58\cdot 10^{-15}$ &$1.81\cdot 10^{-15}$ &$2.04\cdot 10^{-15}$ \\
\cline{2-6}
&$e^{Q}_{\SP}$ &$1.87\cdot 10^{-14}$  &$2.12\cdot 10^{-12}$ &$2.82\cdot 10^{-6}$ & $3.06\cdot 10^{-2}$\\ 
\cline{2-6}
&\text{$\#$ of \eqref{eq:qdwh_qr_reccur}} &$1$ &$2$ &$2$ &$3$\\
\hline
 \end{tabular}
 }
\caption{Comparison of errors in the QDWH algorithm with one or several QR-based iterations.}
 \label{table:iterations_vs_error}
\end{table}
As can be seen in Table~\ref{table:iterations_vs_error}, the errors obtained by both variants of the QDWH algorithm exhibit a similar behavior. Even for tiny spectral gaps, 
no significant loss  of accuracy is observed if only one QR-based iteration is performed.
\end{example}

\section{Hierarchical matrix approximation of spectral projectors}
\label{hmatrices}

Introduced in the context of integral and partial differential equations,
hierarchical matrices allow for the data-sparse representation of a certain class of dense matrices. In the following, we briefly recall the concept of hierarchical matrices and 
some operations; see, e.g.,~\cite{Bebendorf2008,Hackbusch2015} for more details.

\subsection{Matrices with hierarchical low-rank structures}

\subsubsection{HODLR matrices}

We first discuss \textit{hierarchically off-diagonal low-rank} (HODLR) matrices. For convenience, 
we assume that $n = 2^p$ for $p \in \mathbb{N}$. Given a prescribed maximal off-diagonal rank $k\in \mathbb{N}$, we suppose that a matrix $M\in \R^{n\times n}$
admits the representation
\begin{equation}
 \label{eq:HODLR}
 M = \begin{bmatrix}
   M^{(1)}_{1} &U_1^{(1)}V_1^{(1)^{*}}\\
   U_2^{(1)}V_2^{(1)^{*}} & M^{(1)}_{2}\\ 
     \end{bmatrix},
\end{equation}
where $M^{(1)}_i \in \R^{\frac{n}{2} \times \frac{n}{2}}, U^{(1)}_i, V^{(1)}_i \in \R^{\frac{n}{2}\times k}$, for $i = 1,2$, and $k\ll n$. A HODLR matrix
is obtained by applying~\eqref{eq:HODLR} recursively to the diagonal blocks $M^{(l-1)}_i$, where $i = 1,\ldots, 2^{l-1}$ for the $l$th level of recursion, $2 \leq l \leq p$. 
The recursion terminates when the diagonal blocks are sufficiently small, that is, $\frac{n}{2^l} \leq n_{\min}$ for a minimal block size $n_{\min} \in \mathbb{N}$; see Figure~\ref{fig:h_hodlr_matrices} below 
for an illustration. Formally, we define the set of HODLR matrices with block-wise rank $k$ as 
\begin{equation*}
  \mathcal{H}(k) := \left\lbrace M \in \mathbb{R}^{n \times n}: \rank M\vert_{\off} \leq k  \text{ $\forall$off-diagonal block $M\vert_{\off}$ in recursive subdivision} \right\rbrace\text{.}
\end{equation*}
Any matrix $M \in \mathcal{H}(k)$ admits a data-sparse representation. By storing the off-diagonal blocks in 
terms of their low-rank factors and the diagonal blocks as dense matrices, the memory required for representing $M$ is $\mathcal{O}(kn\log n)$, assuming that $k$ is constant with respect to $n$.

Given a general matrix $A \in \mathbb R^{n\times n}$, an approximation $M \in \mathcal{H}(k)$ to $A$ is obtained by computing 
truncated singular value decompositions of the off-diagonal blocks of $A$. The quality of such an approximation is governed by the truncated singular values. 
For simplifying the presentation, we have assumed that the ranks in the off-diagonal blocks are all bounded by the same integer $k$. In practice, we choose these ranks adaptively based on an absolute truncation tolerance $\epsilon$ and they may be different for each block.
 
As explained in~\cite{Bebendorf2008,Hackbusch2015}, several matrix operations can be performed approximately and efficiently within the HODLR format.
The use of formatted arithmetics leads to linear-polylogarithmic complexity for these operations. Table~\ref{table:complexity_HODLR} summarizes the complexity of operations needed by the QDWH algorithm 
for $M_1,M_2,R\in \mathcal{H}(k)$, where $T$ is triangular, and $v\in \mathbb{R}^n$.
 \begin{table}[ht!]
 \caption{Complexity of some arithmetic operations in the HODLR format.}
\label{table:complexity_HODLR}
\centering \begin{tabular}{rcc}
 \hline
 Operation  &  & Computational complexity \\ \hline
 Matrix-vector mult. $M_1*_{\mathcal{H}} v$ &  & $\mathcal{O}(kn\log n)$ \\
 Matrix addition $M_1 +_{\mathcal{H}} M_2 \in \mathcal{H}(k)$ &  & $\mathcal{O}(k^2n\log n)$ \\
 Matrix multiplication $M_1 *_{\mathcal{H}} M_2 \in \mathcal{H}(k)$ &  & $\mathcal{O}(k^2n\log^2 n)$ \\
 Cholesky decomposition $\h\operatorname{-Cholesky}(M_1) \in \mathcal{H}(k)$ & & $\mathcal{O}(k^2n\log^2 n)$ \\ 
 Solving triangular system $M_1 *_{\mathcal{H}} T = M_2 \in \mathcal{H}(k)$  & & $\mathcal{O}(k^2n\log^2 n)$ \\ \hline
 \end{tabular}
\end{table}

\begin{remark} 
\label{remark:qr}
The QR-based iteration~\eqref{eq:qdwh_qr_reccur} of QDWH requires the computation of the QR decomposition~\eqref{eq:qdwh_qr_decomposition}.
Unlike for $\h$-Cholesky, there is no straightforward way of performing QR decompositions in  hierarchical matrix arithmetics. To our knowledge, three different algorithms~\cite{Bebendorf2008,BennerMach2010,Lintner2004} have been proposed for this purpose.
However, each of them seems to have some drawbacks, e.g., failing to achieve a highly accurate decomposition or leading to loss of orthogonality in the orthogonal factor. Hence, instead of using any of the existing algorithms, we develop a 
novel method in Section~\ref{qr-based-iteration} to compute the QR decomposition~\eqref{eq:qdwh_qr_reccur} that exploits the particular structure of the matrix in the first iteration of the QDWH algorithm. 
\end{remark}

\subsubsection{Hierarchical matrices} \label{sec:hierarchicalmatrices}

Let $I = \lbrace 1, 2,\ldots, n\rbrace$ denote the row and column index sets of a matrix $M\in \R^{n\times n}$. To consider more general hierarchical matrices, we define a partition $P$ of $I \times I$ as follows. On level $l = 0$, the index set $I^0 := I$ 
is partitioned into $I^0 = I_1^1 \cup I_2^1$, with $I_1^1 = \lbrace 1,\ldots,\frac{n}{2} \rbrace$ and $I^1_2 = \lbrace \frac{n}{2}+1,\ldots,n \rbrace$. 
At this point, the partition $P$ contains five blocks: $I\times I$ and $I^1_i\times I^1_j$ for $i,j = 1,2$. The subdivision continues as follows: on each level $l = 1,\ldots, p-1$ the index sets $I^l_i$ are 
partitioned into sets $I^{l+1}_{2i-1}$ and $I^{l+1}_{2i}$ of equal size, contributing the blocks $I^{l+1}_i\times I^{l+1}_j$ for $i,j = 1,\ldots, 2^l$ to the partition $P$. The recursion terminates when a 
block $I^l_i\times I^l_j$ satisfies a certain admissibility condition or when $\min\lbrace \vert I^l_i\vert, \vert I^l_j\vert\rbrace \leq n_{\min}$ holds.
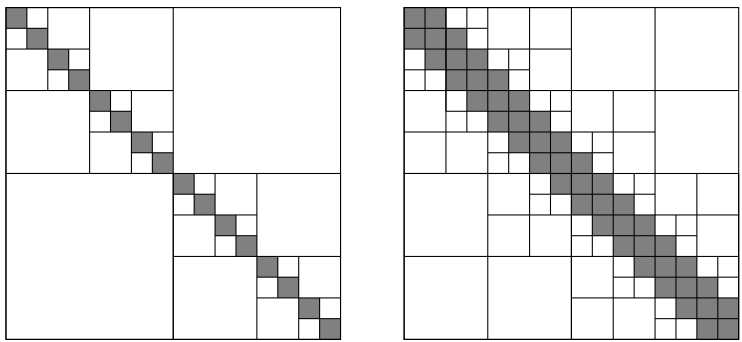
\begin{figure}[ht!]
\centering
\begin{tikzpicture}[scale=0.55]
\fill[gray] (0.5,7)--(1,7)--(1,7.5)--(0.5,7.5);
\fill[gray] (0,8)--(0,7.5)--(0.5,7.5)--(0.5,8);

\fill[gray] (1.5,6)--(2,6)--(2,6.5)--(1.5,6.5);
\fill[gray] (1,7)--(1,6.5)--(1.5,6.5)--(1.5,7);

\fill[gray] (2.5,5)--(3,5)--(3,5.5)--(2.5,5.5);
\fill[gray] (2,6)--(2,5.5)--(2.5,5.5)--(2.5,6);

\fill[gray] (3.5,4)--(4,4)--(4,4.5)--(3.5,4.5);
\fill[gray] (3,5)--(3,4.5)--(3.5,4.5)--(3.5,5);

\fill[gray] (4.5,3)--(5,3)--(5,3.5)--(4.5,3.5);
\fill[gray] (4,4)--(4,3.5)--(4.5,3.5)--(4.5,4);

\fill[gray] (5.5,2)--(6,2)--(6,2.5)--(5.5,2.5);
\fill[gray] (5,3)--(5,2.5)--(5.5,2.5)--(5.5,3);

\fill[gray] (6.5,1)--(7,1)--(7,1.5)--(6.5,1.5);
\fill[gray] (6,2)--(6,1.5)--(6.5,1.5)--(6.5,2);

\fill[gray] (7.5,0)--(8,0)--(8,0.5)--(7.5,0.5);
\fill[gray] (7,1)--(7,0.5)--(7.5,0.5)--(7.5,1);

\draw (0,0) rectangle (8,8);
\draw (0,0) rectangle  (4, 4);
\draw (4,0) rectangle  (8, 4);
\draw (0,4) rectangle  (4, 8);

\draw (0,6)--(4,6);
\draw (2,4)--(2,8);

\draw (4,2)--(8,2);
\draw (6,0)--(6,4);

\draw (0,7)--(2,7);
\draw (1,6)--(1,8);

\draw (0,7.5)--(1,7.5);
\draw (0.5,7)--(0.5,8);

\draw (1,6.5)--(2,6.5);
\draw (1.5,6)--(1.5,7);

\draw (3,4)--(3,6);
\draw (2,5)--(4,5);

\draw (2,5.5)--(3,5.5);
\draw(2.5,5)--(2.5,6);

\draw (3,4.5)--(4,4.5);
\draw (3.5,4)--(3.5,5);

\draw (4,3)--(6,3);
\draw (5,2)--(5,4);

\draw (4,3.5)--(5,3.5);
\draw (4.5,3)--(4.5,4);

\draw[] (5,2.5)--(6,2.5);
\draw (5.5,2)--(5.5,3);

\draw (6,1)--(8,1);
\draw (7,0)--(7,2);

\draw (6,1.5)--(7,1.5);
\draw (6.5,1)--(6.5,2);

\draw (7,0.5)--(8,0.5);
\draw (7.5,0)--(7.5,1);
\end{tikzpicture} 
\qquad
\begin{tikzpicture}[scale=0.55]
\fill[gray] (0,8)--(0,7)--(1,7)--(1,8);
\fill[gray] (1,7)--(1.5,7)--(1.5,7.5)--(1,7.5);
\fill[gray] (1,7)--(1,6.5)--(0.5,6.5)--(0.5,7);

\fill[gray] (1,7)--(1,6)--(2,6)--(2,7);
\fill[gray] (2,6)--(2.5,6)--(2.5,6.5)--(2,6.5);
\fill[gray] (2,6)--(2,5.5)--(1.5,5.5)--(1.5,6);

\fill[gray] (2,6)--(2,5)--(3,5)--(3,6);
\fill[gray] (3,5)--(3.5,5)--(3.5,5.5)--(3,5.5);
\fill[gray] (3,5)--(3,4.5)--(2.5,4.5)--(2.5,5);

\fill[gray] (3,5)--(3,4)--(4,4)--(4,5);
\fill[gray] (4,4)--(4.5,4)--(4.5,4.5)--(4,4.5);
\fill[gray] (4,4)--(4,3.5)--(3.5,3.5)--(3.5,4);

\fill[gray] (4,4)--(4,3)--(5,3)--(5,4);
\fill[gray] (5,3)--(5.5,3)--(5.5,3.5)--(5,3.5);
\fill[gray] (5,3)--(5,2.5)--(4.5,2.5)--(4.5,3);

\fill[gray] (5,3)--(5,2)--(6,2)--(6,3);
\fill[gray] (6,2)--(6.5,2)--(6.5,2.5)--(6,2.5);
\fill[gray] (6,2)--(6,1.5)--(5.5,1.5)--(5.5,2);

\fill[gray] (6,2)--(6,1)--(7,1)--(7,2);
\fill[gray] (7,1)--(7.5,1)--(7.5,1.5)--(7,1.5);
\fill[gray] (7,1)--(7,0.5)--(6.5,0.5)--(6.5,1);

\fill[gray] (7,1)--(7,0)--(8,0)--(8,1);

\draw (0,0) rectangle (8,8);
\draw (0,0) rectangle  (4, 4);
\draw (4,0) rectangle  (8, 4);
\draw (0,4) rectangle  (4, 8);

\draw (0,2)--(8,2);
\draw (0,6)--(8,6);

\draw (2,0)--(2, 8);
\draw (6,0)--(6,8);

\draw (2,3)--(8,3);
\draw (4,1)--(8,1);
\draw (5,0)--(5,6);
\draw (7,0)--(7,4);

\draw (0,7)--(4,7);
\draw (0,5) -- (6,5);
\draw (3,2)--(3,8);
\draw (1,4)--(1,8);

\draw (0,7.5)--(2,7.5);
\draw (0.5, 6)--(0.5,8);
\draw (1.5,5)--(1.5,8);

\draw (0,6.5)--(3,6.5);
\draw (1,5.5)--(1,8);
\draw (2.5, 4)--(2.5,7); 

\draw (1,5.5)--(4,5.5);
\draw (2,4.5)--(5,4.5);
\draw (3.5,3)--(3.5,6);
\draw (4.5,2)--(4.5,5);

\draw (3,3.5)--(6,3.5);
\draw(4,2.5)--(7,2.5);

\draw (5.5,1)--(5.5,4);
\draw (5,1.5)--(8,1.5);
\draw (6.5,0)--(6.5,3);
\draw (6,0.5)--(8,0.5);
\draw (7.5,0)--(7.5, 2);
\end{tikzpicture}
\caption{Left: HODLR matrix. Right: $\h$--matrix with admissibility condition~\eqref{eq:stand_adm}. Blocks colored grey are stored as dense matrices.}
\label{fig:h_hodlr_matrices}
\end{figure}%

Inspired by discretizations for 1D integral equations~\cite{Hackbusch1999}, we make use of the following admissibility condition:
\begin{equation}
\label{eq:stand_adm} 
\text{block } \tau = t \times s\text{ is admissible } \Longleftrightarrow\ \min\{\diam(t), \diam(s) \}\leq  \dist(t,s)\text{,}
\end{equation}
with $$\diam(t):= \underset{i,j\in t} \max \hskip 3pt \vert i-j\vert, \quad \dist(t,s):= \underset{i\in t, j\in s} \min\vert i-j\vert.$$ 
See Figure~\ref{fig:h_hodlr_matrices} for an illustration of the resulting partition $P$. Given $P$, the set of $\h$--matrices with block-wise rank $k$ is defined as 
\begin{equation*}
 \mathcal{H}(P,k) := \left\lbrace M \in \mathbb{R}^{n \times n}: \rank M\vert_{\tau} \leq k \text{ for all admissible blocks $\tau \in P$} \right\rbrace\text{.}
\end{equation*}%
The complexity of arithmetic operations displayed in Table~\ref{table:complexity_HODLR} extends to $\h(P,k)$.   

\begin{example}  \label{ex:hodlrvshmatrix} \rm
We investigate the potential of the HODLR and $\h$--matrix formats to efficiently store spectral projectors of banded matrices. For this purpose, we have generated, as explained in Section~\ref{sec:construct_test_matrices}, a symmetric $b$-banded matrix $A\in \mathbb{R}^{16000\times 16000}$ with
eigenvalues in $[-1,\hskip 3pt -\gap] \cup [\gap, \hskip 3pt 1]$. The memory 
needed to store the full spectral projector $\Pi_{<0}(A)$ in double precision is $2048$ MB. We choose $n_{\min} = 250$, a truncation tolerance $\epsilon = 10^{-10}$, and $\gap \in \{10^{-1},10^{-4}\}$. Table~\ref{table:storage_spec} reveals that the HODLR format often requires less memory to approximately store $\Pi_{<0}(A)$, unless both $\gap$ and the bandwidth are large. In terms of computational time, the outcome is even clearer.
For bandwidth $b = 8$ and $\gap = 10^{-1}$, a situation that favors the $\h$--matrix format in terms of memory, we have run the algorithm described in Section~\ref{overall_algorithm} in both formats. It turned out that the use of the HODLR format
led to an overall time of $608$ seconds, while the $\h$--matrix format required $792$ seconds.
\end{example}

\begin{table}[h!]
\caption{Memory required to approximately store spectral projectors for the banded matrices from Example~\ref{ex:hodlrvshmatrix} in HODLR and $\h$--matrix format.}
 \label{table:storage_spec}
\centering
\begin{tabular}[t]{c||c||c}
$\gap = 10^{-1}$ &HODLR &\text{$\h$--matrix} \\
 \hline
 \hline
 b = 1 &$55.72$ MB & $95.16$ MB\\
 b = 2 &$79.38$ MB &$96.42$ MB\\
 b = 4 &$127.04$ MB  &$106.54$ MB\\
 b = 8  &$219.92$ MB &$151.06$ MB \\
 b = 16 &$395.91$ MB &$291.85$ MB \\
 \hline
\end{tabular}
\quad
\begin{tabular}[t]{c||c||c}
$\gap  = 10^{-4}$ &HODLR &\text{$\h$--matrix} \\
 \hline
 \hline
 b = 1 &$86.03$ MB & $128.58$ MB\\
 b = 2 &$129.71$ MB &$160.56$ MB\\
 b = 4 &$206.32$ MB  &$225.72$ MB\\
 b = 8  &$340.88 $ MB &$352.54$ MB \\
 b = 16 &$567.69$ MB &$583.93$ MB \\
 \hline
\end{tabular}
\end{table}

Based on the evidence provided by Example~\ref{ex:hodlrvshmatrix}, we have concluded that more general $\h$--matrix formats bring little advantage and thus focus on the HODLR format for the rest of this paper.

\subsection{A priori bounds on singular values and memory requirements}

To study the approximation of $\Pi_{<0}(A)$ in the HODLR format, we first derive bounds for the singular values of the off-diagonal ranks based on rational approximations to the $\sign$ function.  In the 
following, we say that a rational function $r$ is of type $(k,s)$ and write $r\in \mathcal{R}_{k,s}$ if $r = p/q$ holds for polynomials $p$ and $q$ of degree at most $k$ and $s$, respectively.
  
\subsubsection{Rational approximation of sign function}
\label{sec:rational_approx}
Given $R>0$, the min-max problem 
\begin{equation}
 \label{eq:min_max_rational}
 \underset{r\in \mathcal{R}_{2m-1, 2m}}\min \underset{x \in [-R,-1]\cup [1,R]}\max \vert \sign(x) - r(x)\vert
\end{equation}
has a unique solution $s_{m}$. Called a Zolotarev function of type $(2m-1, 2m)$ corresponding to $R$ (see e.g.~\cite[Chapter 9]{Akhiezer1990}), this function takes the form
\begin{equation*}
 s_{m}(x) := Cx \frac{\prod_{i = 1}^{m-1} (x^2 +c_{2i})}{ \prod_{i = 1}^{m} (x^2 +c_{2i-1})}.
\end{equation*}
The coefficients $c_i, i = 1,\ldots,2m$ are given in terms of the Jacobi elliptic function $\operatorname{sn}(\cdot;\kappa)$:
\begin{equation}
 \label{eq:zolo_coeff}
 c_i = \frac{\sn^2(\frac{iK(\kappa)}{2m};\kappa)}{1 - \sn^2(\frac{iK(\kappa)}{2m};\kappa)},
\end{equation}
where $\kappa = \sqrt{1 - 1/R^2}$ and $K(\kappa)$ is defined as the complete elliptic integral of the first kind 
$$K(\kappa) = \int_{0}^{\frac{\pi}{2}} \frac{d\theta}{\sqrt{1 - \kappa^{2}\sin^2 \theta}} =\int_{0}^{1} \frac{dt}{\sqrt{(1-t^2)(1 - \kappa^2t^2)}}.$$ 
The constant $C$ is uniquely determined by the condition
\begin{equation*}
\underset{x \in [-R,-1]}\min 1 + s_{m}(x) =  \underset{x \in [1,R]}\max 1 - s_{m}(x).
 \end{equation*}
 
As shown in~\cite{GuetPoliTang2015}, the approximation error $E_{m} := \underset{x \in [-R,-1]\cup [1,R]}\max \vert \sign(x) - s_{m}(x)\vert$ is bounded as 
\begin{equation}
 \label{eq:error_sign_rational}
  \frac{4\rho^{m}}{\rho^{m} +1} \leq E_{m} \leq 4\rho^{m}, 
\end{equation}
where $\rho = \rho(\mu) = \exp\big(-\frac{\pi K(\mu')}{2 K(\mu)}\big)$ with $\mu = \big(\frac{\sqrt{R}-1}{\sqrt{R}+1}\big)^{2}$ and $\mu' = \sqrt{1 - \mu^2}$.
The following lemma derives a bound from~\eqref{eq:error_sign_rational} that reveals the influence of the gap on the error.
\begin{lemma} \label{lemma:approxerror}
 With the notation introduced above and $\gap = 1/R$, it holds that 
\begin{equation}
 \label{eq:upper_bound_simplified}
 E_{m} \leq 4 \exp\bigg(-\frac{\pi^2m}{ 4\log \big( 4/ \sqrt[4]{\gap} + 2\big)}\bigg).
\end{equation}
\end{lemma}
\begin{proof}
Following Braess and Hackbusch~\cite{BraessHack2005}, we have
$$K(\mu') \geq \pi/2, \hskip 10pt K(\mu) \leq \log (4/ \mu' + 2)\text{.}$$
Thus, the upper bound in~\eqref{eq:error_sign_rational} implies
\begin{equation*}
 E_{m} \leq 4\exp\bigg(-\frac{\pi^2m}{ 4\log ( 4/ \mu' + 2 )}\bigg). 
\end{equation*}
From 
\[
 \mu' = \sqrt{1 - \left(\frac{1 -\sqrt{\gap}}{1+ \sqrt{\gap}}\right)^{4}} = \frac{\sqrt{8\sqrt{\gap}(1+\gap)}}{(1+\sqrt{\gap})^2} \geq \sqrt[4]{\gap}
\]
it follows that $\log ( 4/\mu' + 2 ) \leq  \log ( 4/ \sqrt[4]{\gap} + 2)$, which completes the proof.
\end{proof}

It is simple to bound the ranks of the off-diagonal blocks for a rational function applied to a banded matrix. 

\begin{lemma}
\label{lemma:rational_ranks}
 Consider a $b$-banded matrix $A \in \mathbb{R}^{n\times n}$ and a rational function $r_{m}$ of type $(2m-1, 2m)$, with poles disjoint from the spectrum of $A$. Then the off-diagonal blocks of 
 $r_{m}(A)$ have rank at most $2mb$.  
\end{lemma} 
\begin{proof} 
Assuming that $r$ has simple poles, let $r_{m}(x) = \sum_{i=1}^{2m} \omega_i (x - \mu_i)^{-1}$ be a partial fraction expansion of $r_{m}$, with $\omega_i, \mu_i\in\mathbb{C}, i = 1,\ldots, 2m$. 
Thus, $r_{m}(A)$ is a sum of $2m$ shifted inverses of $A$. By a well known result (see, e.g.,~\cite{VandeVanBarelMastro2008}), 
the off-diagonal blocks of each summand $B = A-\mu_i I$ satisfy $\rank B^{-1}|_{\off} = \rank B |_{\off}$. Noting that $\rank B |_{\off}=b$, because $B$ has bandwidth $b$, this completes the proof for simple poles. The result extends to non-simple poles by the semi-continuity of the rank function.
\end{proof}

\subsubsection{Singular value decay of off-diagonal blocks}

The results of Lemma~\ref{lemma:approxerror} and  Lemma~\ref{lemma:rational_ranks} allow us to establish exponential decay 
for the singular values of the off-diagonal blocks in $\Pi_{<0}(A)$ or, equivalently, in $\sign(A)$ for any symmetric banded matrix $A$. By rescaling $A$, we may assume without loss of generality that its 
spectrum is contained in $[-R, \hskip 3pt -1] \cup [1, \hskip 3pt R]$. We let $\sigma_i(\cdot)$ denote the $i$th largest singular value of a matrix.
\begin{theorem}
\label{sign_sing_value_decay}
Consider a symmetric $b$-banded matrix $A\in\mathbb{R}^{n\times n}$ with the eigenvalues contained in $[-R, \hskip 3pt -1] \cup [1, \hskip 3pt R]$, 
and $m \in \mathbb{N}$. Letting $\gap = 1/R$, the singular values of any off-diagonal block $\Pi_{<0}(A)\vert_{\off}$ satisfy
\begin{equation*}
 \sigma_{2mb+1}(\Pi_{<0}(A)\vert_{\off}) \leq 2\exp\bigg(-\frac{\pi^2m}{ 4\log (4/\sqrt[4]{\gap} + 2)}\bigg).
\end{equation*}
\end{theorem} 
\begin{proof}
Let $s_{m}$ denote the solution of the min-max problem~\eqref{eq:min_max_rational}.
Because
$s_{m}(A)\vert_{\off}$ has rank at most $2mb$ by Lemma~\ref{lemma:rational_ranks}, and the best rank-$i$ approximation error is governed by the $(i+1)$th largest singular value, it follows 
from~\eqref{eq:upper_bound_simplified} that
\begin{align*}
\sigma_{2mb+1}(\sign(A)\vert_{\off}) &\le \Vert \sign(A) - s_{m}(A) \Vert_2  \leq \underset{x \in [-R,-1]\cup [1,R]}\max \vert \sign(x) - s_{m}(x)\vert \notag \\
  &\leq 4\exp\bigg(-\frac{\pi^2m}{ 4\log (4/\sqrt[4]{\gap} + 2)}\bigg)\text{.}
\end{align*}
The statement thus follows from the relation $\Pi_{<0}(A)\vert_{\off} = -\frac{1}{2}\sign(A)\vert_{\off}$.
\end{proof}

\subsubsection{Memory requirements with respect to gap}

Theorem~\ref{sign_sing_value_decay} allows us to study the memory required to approximate $\Pi_{<0}(A)$ in the HODLR format to a prescribed accuracy. For this purpose, 
let $\Pi^{\h}$ denote the best approximation in the Frobenius norm of $\Pi_{<0}(A)$ in the HODLR format with all off-diagonal ranks bounded by $2mb$. 
Necessarily, the diagonal blocks of $\Pi^{\h}$ and $\Pi_{<0}(A)$ are the same. For an off-diagonal block of size $k$, Theorem~\ref{sign_sing_value_decay} implies
\begin{align*}
 \Vert \Pi_{<0}(A)|_{\off} - \Pi^{\h}|_{\off} \Vert_F^2 & = \sum_{i = 2mb+1}^k \sigma_{i}(\Pi_{<0}(A)\vert_{\off})^2 \le 
 \sum_{j = m}^{\lceil k/2b \rceil -m } 2 b\, \sigma_{2j b+1}(\Pi_{<0}(A)\vert_{\off})^2 \\
 &\le 8b \sum_{j = m}^{\lceil k/2b \rceil -m } \tau^{2j} \le \frac{8b}{1-\tau^2} \tau^{2m},
\end{align*}
with $\tau = \exp\Big(-\frac{\pi^2}{ 4\log (4/\sqrt[4]{\gap} + 2)}\Big)$. Taking into account the total number of off-diagonal blocks, we arrive at
\[
 \Vert \Pi_{<0}(A) - \Pi^{\h} \Vert_F^2 \le \frac{16 b}{1-\tau^2} (n/n_{\min} - 1) \tau^{2m}
\]
Thus, the value of $m$ needed to attain $\Vert \Pi_{<0}(A) - \Pi^{\h} \Vert_F\le \delta$ for a 
desired accuracy $\delta > 0$ satisfies $m = \mathcal{O}\big( \vert\log \gap \vert \cdot \log\big(bn \delta^{-1} \vert\log \gap \vert\big) \big)$. 

The corresponding approximation $\Pi^{\h}$ requires
\begin{equation}
\label{eq:memory_sproj_rational}
\mathcal{O}\Big(\vert\log \gap \vert \cdot \log\big(bn \delta^{-1} \vert\log \gap \vert\big) bn \log n \Big)
\end{equation}
memory. Up to a double logarithmic factor, this shows that the memory depends logarithmically on the spectral gap.  

\subsubsection{Comparison to approximate sparsity}

We now compare~\eqref{eq:memory_sproj_rational} with known results for approximate sparsity.
Assuming we are in the setting of Theorem~\ref{sign_sing_value_decay}, it is shown in~\cite{BenBoiRaz2013} that the off-diagonal entries of $\Pi_{<0}(A)$ satisfy
$$\vert (\Pi_{<0}(A))_{ij} \vert \leq C \e^{-\alpha \vert i-j \vert}, \qquad \alpha = \frac{1}{2b} \log\left(\frac{1+\gap}{1-\gap}\right),$$
for some constant $C>0$ depending only on $R$.

Let $\Pi^{(m)}$ denote the best approximation in the Frobenius norm to $\Pi_{<0}(A)$ by a matrix of bandwidth $m$. Following~\cite[Theorem 7.7]{BenBoiRaz2013}, we obtain
\[
 \Vert \Pi_{<0}(A) - \Pi^{(m)} \Vert_{F} \leq \frac{C}{\sqrt{\alpha}}\sqrt{n} \e^{-\alpha m}\text{.}
\]

Choosing a value of $m$ that satisfies $m = \mathcal{O}\big( b \gap^{-1} \log\big(C bn \delta^{-1} \gap^{-1}\big) \big)$ thus ensures 
an accuracy of $\delta > 0$, where we used $\alpha \approx \gap / b$.
Since the storage of $\Pi^{(m)}$ requires $\mathcal{O} (mn)$ memory, we arrive at 
\begin{equation} 
\label{eq:memory_sproj_poly}
\mathcal{O}\left(\frac{1}{\gap}  \log\big(C bn \delta^{-1} \gap^{-1}\big)  bn\right)
\end{equation}
memory. In contrast to the logarithmic dependence in~\eqref{eq:memory_sproj_rational}, the spectral gap now enters the asymptotic complexity inversely proportional.
On the other hand,~\eqref{eq:memory_sproj_rational} features a factor $\log n$ that is not present in~\eqref{eq:memory_sproj_poly}.
For most situations of practical interest, we expect that the much milder dependence on the gap far outweighs this additional factor. In summary, 
the comparison between~\eqref{eq:memory_sproj_rational} and~\eqref{eq:memory_sproj_poly} provides strong theoretical justification for favoring the HODLR format over approximate sparsity.

\section{QR-based first iteration of QDWH}
\label{qr-based-iteration}

The first QR-based iteration of the QDWH algorithm requires computing the QR decomposition
\begin{equation} \label{eq:qrdecomp}
 \begin{bmatrix}
  c A \\
  I
 \end{bmatrix}  =  
 \begin{bmatrix}
  Q_1 \\
  Q_2
 \end{bmatrix}R 
\end{equation}
for some scalar $c>0$. Without loss of generality, we suppose that $c = 1$. In this section, we develop an algorithm that requires $\mathcal O(b^2 n)$ operations for 
performing this decomposition when $A$ is a $b$--banded matrix. In particular, our algorithm directly computes $Q_1$ and $Q_2$ in the HODLR format. 
Since it is significantly simpler, we first discuss the case of a tridiagonal matrix $A$ before treating the case of general $b$.

It is interesting to note that the need for computing a QR decomposition of the form~\eqref{eq:qrdecomp} also arises in the solution of ill-posed inverse problems with 
Tikhonov regularization; see, e.g.,~\cite{Bjorck1996}. However, when solving ill-posed problems, usually only the 
computation of the upper-triangular factor $R$ is required, while the QDWH algorithm requires the computation of the orthogonal factor.

\subsection[QR decomposition for tridiagonal A]{QR decomposition of $\big[ {A \atop I} \big]$ for tridiagonal $A$}

\label{qr_tridiag}

For the case of a bidiagonal matrix $A$, Eld{\'e}n~\cite{Elden1977} proposed a fast algorithm for reducing a matrix $\big[ {A \atop I} \big]$ to upper triangular form. In the following, we 
propose a modification of Eld{\'e}n's algorithm suitable for tridiagonal $A$.
\begin{figure}[ht!]
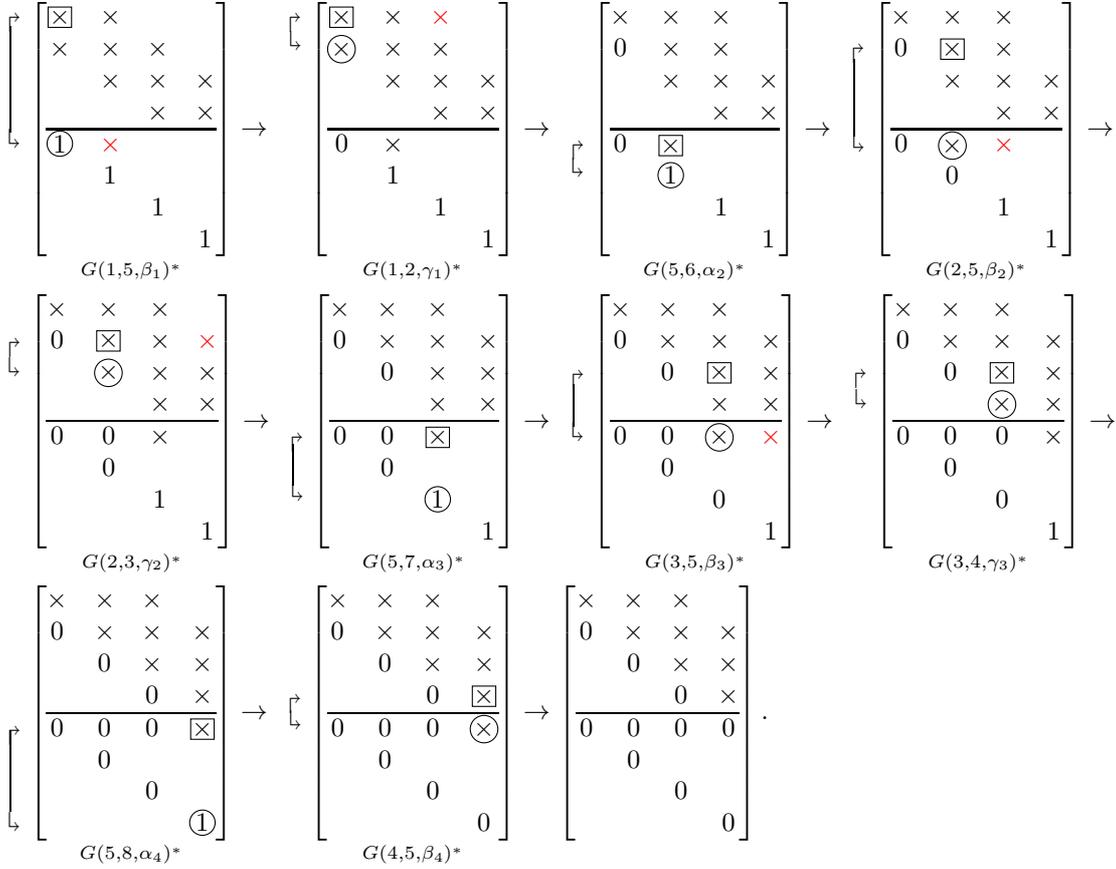
 
\begin{align*}
 &\begin{array}{c@{\hspace{1mm}}}
   \Rc\\[-0.4cm]
       \\
  \clr \\
  \clr \\ 
  \clr \\
  \rc \\
      \\
      \\   
      \\ 
 \end{array}
 \underset{G(1,5,\beta_1)^{*}}{
 \begin{lbmatrix}{4}
 \circledr{$\times$} &\times \\
 \times &\times &\times\\
 &\times &\times &\times \\
 & &\times &\times\\ \hline
\circled{1}  &\textcolor{red}{\times} & &\\
 &1 & &\\
 & &1 &\\
  & & &1\\
\end{lbmatrix}}\rightarrow 
 \begin{array}{c@{\hspace{1mm}}}
   \Rc  \\
   \rc  \\[-0.4cm]
        \\
        \\ 
        \\
        \\
        \\
        \\   
        \\ 
 \end{array}
\underset{G(1,2,\gamma_1)^{*}}{\begin{lbmatrix}{4}
\circledr{$\times$} &\times &\textcolor{red}{\times} \\
\circledt{$\times$} &\times &\times\\
 &\times &\times &\times \\
 & &\times &\times\\ \hline
0  &\times & &\\
 &1 & &\\
 & &1 &\\
  & & &1\\
\end{lbmatrix}}\rightarrow  
\begin{array}{c@{\hspace{1mm}}}
        \\
        \\
        \\
        \\ 
  \Rc   \\[-0.01cm]
  \rc   \\[-0.4cm]
        \\
        \\   
        \\ 
 \end{array}
\underset{G(5,6,\alpha_2)^{*}}{\begin{lbmatrix}{4}
\times &\times  &\times\\
0 &\times &\times\\
 &\times &\times &\times \\
 & &\times &\times\\ \hline
0  &\circledr{$\times$} & &\\
 &\circled{1} & &\\
 & &1 &\\
  & & &1\\
\end{lbmatrix}} \rightarrow  
 \begin{array}{c@{\hspace{1mm}}}
        \\
  \Rc   \\[0.05cm]
  \clr  \\ [0.05cm]
   \clr \\ 
  \rc   \\[-0.05cm]
        \\
        \\
        \\    
 \end{array}
\underset{G(2,5,\beta_2)^{*}}{\begin{lbmatrix}{4}
\times &\times  &\times\\
0 &\circledr{$\times$} &\times\\
 &\times &\times &\times \\
 & &\times &\times\\ \hline
0  &\circledt{$\times$} &\textcolor{red}{\times} &\\
 &0 & &\\
 & &1 &\\
  & & &1\\
\end{lbmatrix} }\rightarrow \\
 &\begin{array}{c@{\hspace{1mm}}}
        \\
   \Rc  \\[0.01cm]
   \rc  \\[-0.02cm]
        \\ 
        \\
        \\
        \\
        \\    
 \end{array}
\underset{G(2,3,\gamma_2)^{*}}{\begin{lbmatrix}{4}
\times &\times  &\times\\
0 &\circledr{$\times$} &\times &\textcolor{red}{\times}\\
 &\circledt{$\times$} &\times &\times \\
 & &\times &\times\\ \hline
0  &0 &\times &\\
 &0 & &\\
 & &1 &\\
  & & &1\\
\end{lbmatrix}} \rightarrow 
\begin{array}{c@{\hspace{1mm}}}
        \\
        \\
        \\
        \\ 
  \Rc   \\[-0.01cm]
  \clr  \\
  \rc   \\[-0.4cm]
        \\   
        \\ 
 \end{array}
\underset{G(5,7,\alpha_3)^{*}}{\begin{lbmatrix}{4}
\times &\times  &\times\\
0 &\times &\times &\times\\
 &0 &\times &\times \\
 & &\times &\times\\ \hline
0  &0 &\circledr{$\times$}&\\
 &0 & &\\
 & &\circled{1} &\\
  & & &1\\
\end{lbmatrix}}\rightarrow 
\begin{array}{c@{\hspace{1mm}}}
        \\
        \\
  \Rc   \\ [-0.4cm]
        \\ 
  \clr  \\ 
  \rc   \\[-0.05cm]
        \\
        \\
        \\
\end{array}
\underset{G(3,5,\beta_3)^{*}}{
\begin{lbmatrix}{4}
\times &\times  &\times\\
0 &\times &\times &\times\\
 &0 &\circledr{$\times$} &\times \\
 & &\times &\times\\ \hline
0  &0 &\circledt{$\times$} &\textcolor{red}{\times}\\
 &0 & &\\
 & &0 &\\
  & & &1\\
\end{lbmatrix}}\rightarrow 
\begin{array}{c@{\hspace{1mm}}}
        \\
        \\
  \Rc   \\ [-0.4cm]
        \\ 
   \rc  \\ 
        \\
        \\
        \\
        \\
\end{array}
\underset{G(3,4,\gamma_3)^{*}}{\begin{lbmatrix}{4}
\times &\times  &\times\\
0 &\times &\times &\times\\
 &0 &\circledr{$\times$} &\times \\
 & &\circledt{$\times$} &\times\\ \hline
0  &0 &0 &\times\\
 &0 & &\\
 & &0 &\\
  & & &1\\
\end{lbmatrix}}\rightarrow\\ 
&\begin{array}{c@{\hspace{1mm}}}
        \\
        \\
        \\
        \\ 
  \Rc   \\[-0.1cm]
  \clr  \\[0.1cm]
  \clr  \\ [0.05cm]
  \rc   \\[-0.5cm]
        \\   
 \end{array}
\underset{G(5,8,\alpha_4)^{*}}{\begin{lbmatrix}{4}
\times &\times  &\times\\
0 &\times &\times &\times\\
 &0 &\times &\times \\
 & &0 &\times\\ \hline
0  &0 &0 &\circledr{$\times$}\\
 &0 & &\\
 & &0 &\\
  & & &\circled{1}\\
\end{lbmatrix}}\rightarrow 
\begin{array}{c@{\hspace{1mm}}}
        \\
        \\
        \\ 
   \Rc  \\ [-0.45cm]
        \\ 
  \rc   \\[-0.01cm]
        \\
        \\
        \\
\end{array}
\underset{G(4,5,\beta_4)^{*}}{\begin{lbmatrix}{4}
\times &\times  &\times\\
0 &\times &\times &\times\\
 &0 &\times &\times \\
 & &0 &\circledr{$\times$}\\ \hline
0  &0 &0 &\circledt{$\times$}\\
 &0 & &\\
 & &0 &\\
  & & &0\\
\end{lbmatrix}}\rightarrow 
\begin{lbmatrix}{4}
\times &\times  &\times\\
0 &\times &\times &\times\\
 &0 &\times &\times \\
 & &0 &\times\\ \hline
0  &0 &0 &0\\
 &0 & &\\
 & &0 &\\
  & & &0\\
 \end{lbmatrix} \text{.}
\end{align*}
\caption{Fast QR decomposition of $\big[ {A \atop I} \big]$ for tridiagonal $A$ and $n = 4$. In each step, a Givens rotation is applied to the rows denoted by the arrows. Crosses denote generically nonzero elements, boxed/circled crosses are used to define Givens rotations, while 
red crosses denote the fill-in during the current operation.}
\label{fig:qr_example}
\end{figure}

Our proposed algorithm is probably best understood from the illustration in Figure~\ref{fig:qr_example} for $n = 4$.  In the $i$th step of the algorithm, all subdiagonal elements in 
the $i$th column of $\big[ {A \atop I} \big]$ are annihilated by performing Givens rotations either with the diagonal element, or with the element $(n+1,i)$. By carefully choosing the order of annihilation, 
only one new nonzero subdiagonal element is created in column $i+1$. The detailed pseudocode of this procedure is provided in Algorithm~\ref{alg:alg1}. We use $G(i,j,\alpha)$ to denote a Givens rotation of angle $\alpha$ that is applied to rows/columns $i$ and $j$.

\begin{algorithm}[ht!]
    \caption{\text{Fast QR decomposition~\eqref{eq:qrdecomp} for tridiagonal $A$}}
    \label{alg:alg1}
    \renewcommand{\algorithmicrequire}{\textbf{Input:}}
    \renewcommand{\algorithmicensure}{\textbf{Output:}}
   \begin{algorithmic}[1]
     
   \REQUIRE Tridiagonal matrix $A$.
   \ENSURE Factors $Q,R$ of a QR decomposition of $\big[ {A \atop I} \big]$. 
   
   \STATE $Q \gets I_{2n}, R \gets \big[ {A \atop I} \big]$.

   \STATE Construct $G(1,n+1,\beta_1)$ to annihilate $R(n+1,1)$.
   \STATE Update $R \gets G(1,n+1,\beta_1)^{*}R$ and $Q \gets QG(1,n+1,\beta_1)$

   \STATE Construct $G(1,2,\gamma_1)$ to annihilate $R(2,1)$.
   \STATE Update $R \gets  G(1,2,\gamma_1)^{*}R$ and $Q \gets QG(1,2,\gamma_1)$.
      
   \vskip 1pt
   \FOR{$i = 2,\ldots,n$} 
        \STATE Construct $G(n+1,n+i,\alpha_i)$ to annihilate $R(n+i,i)$.  
        \STATE Update $R \gets G(n+1,n+i,\alpha_i)^{*}R$ and $Q \gets QG(n+1,n+i,\alpha_i)$. \label{alpha}
      
        \STATE Construct $G(i,n+1,\beta_i)$ to annihilate $R(n+1,i)$.
        \STATE Update $R \gets G(i,n+1,\beta_i)^{*}R$ and $Q \gets QG(i,n+1,\beta_i)$.  
      
        \IF{$i<n$}
        \STATE Construct $G(i,i+1,\gamma_i)$ to annihilate $R(i+1,i)$.
        \STATE Update $R \gets G(i,i+1,\gamma_i)^{*}R$ and $Q \gets QG(i,i+1,\gamma_i)$. 
      \ENDIF
   \ENDFOR

 \end{algorithmic}
\end{algorithm}

Algorithm~\ref{alg:alg1} performs $3n-2$ Givens rotations in total. By exploiting its sparsity in a straightforward manner, only 
$\mathcal{O}(n)$ operations and memory are required to compute the upper triangular factor $R$. 
The situation is more complicated for the orthogonal factor. Since $Q$ is dense, it would require $\mathcal{O}(n^2)$ operations and memory to form $Q$ using Algorithm~\ref{alg:alg1}. 
In the following section, we explain how the low-rank structure of $Q$ can be exploited to reduce this cost to $\mathcal{O}(n)$ as well.

\subsubsection{Ranks of off-diagonal blocks and fast computation of orthogonal factor}

For our purposes, it suffices to compute the first $n$ columns of the $2n\times 2n$ matrix $Q$, that is, the $n\times n$ matrices $Q_1 = Q(1:n,1:n)$ and $Q_2 = Q(n+1:2n,1:n)$.
The order of Givens rotations in Algorithm~\ref{alg:alg1} implies that $Q_1$ is an upper Hessenberg matrix while $Q_2$ is an upper triangular matrix.
The following theorem shows that all off-diagonal blocks of $Q_1,Q_2$ have rank at most two.

\begin{theorem}
\label{thm:ranks_tridiag}
For the orthogonal factor $Q$ returned by Algorithm~\ref{alg:alg1}, it holds that the matrices $Q(1:k,k+1:n)$ and $Q(n+1:n+k,k+1:n)$ have rank at most two for all $1 \leq k < n$.
\end{theorem}
\begin{proof}
We only prove the result for $Q(1:k,k+1:n)$; the proof for $Q(n+1:n+k,k+1:n)$ is analogous.

During steps $1,\ldots,k-1$ of Algorithm~\ref{alg:alg1}, $Q(1:k,k+1:n)$ is not modified and remains zero. In step $k$ of Algorithm~\ref{alg:alg1}, column $k+1$ of $Q$ is modified, while $Q(1:k,k+2:n)$ remains zero. After step $k$ has been completed,
let us set 
\begin{equation}
 \label{eq:span_tridiag}
\mathcal{U} := \operatorname{span} \lbrace Q(1:k,k+1), Q(1:k,n+1)\rbrace \subset \R^k.
\end{equation}
By construction, $\myspan Q(1:k,k+1:n) \subset \calU$. In the following, we show by induction that this relation holds for all subsequent steps of Algorithm~\ref{alg:alg1}. 
Suppose that $\myspan Q(1:k,k+1:n) \subset \calU$ holds after $i$ steps for some $i$ with $k\le i \le n-1$. In step $i+1$, the following operations are performed:
\begin{enumerate}
\item  $G(n+1,n+i+1,\alpha_{i+1})$ is applied to columns $n+1$ and $n+i+1$ of $Q$. Because $Q(1:k, n+i+1)$ is zero before applying the rotation, this simply effects a rescaling of column $n+1$ and thus $Q(1:k, n+1)\in \mathcal{U}$ remains true.
\item  $G(i+1,n+1,\beta_{i+1})$ is applied to columns $i+1$ and $n+1$ of $Q$, which preserves $\myspan Q(1:k,k+1:n) \subset \calU$.
\item  If $i<n$, $G(i+1,i+2,\gamma_{i+1})$ is applied to columns $i+1$ and $i+2$ of $Q$, which again preserves $\myspan Q(1:k,k+1:n) \subset \calU$.
\end{enumerate}
After completion of the algorithm, the column span of $Q(1:k,k+1:n)$ is thus contained in a subspace of dimension at most two. This proves the statement of the theorem.
\end{proof}

\begin{remark}
\label{construct_off}
The proof of Theorem~\ref{thm:ranks_tridiag} can be turned into a procedure for directly computing low-rank representations for the off-diagonal blocks of $Q_1, Q_2$ in the HODLR format. Due to the structure of $Q_1$ and $Q_2$, 
all lower off-diagonal blocks have ranks $1$ and $0$ respectively, and the computation of their low-rank representations is straightforward. In the following, we therefore only discuss the computation of a low-rank representation for an upper off-diagonal $p \times s$ block
$Q_1|_{\off} = U_1V_1^{*} $ with $U_1\in\mathbb{R}^{p\times 2}$, $V_1\in\mathbb{R}^{s\times 2}$.

Let $r+1$ and $k+1$ denote the row and column in $Q_1$ that correspond to the first row and column of $Q_1|_{\off}$, respectively. The construction of 
$Q_1|_{\off}$ begins in step $k$ of Algorithm~\ref{alg:alg1}, because $Q_1|_{\off}$ is zero before step $k$. During step $k$ only the first column of $Q_1|_{\off}$ is affected
by $G(k,k+1,\gamma_k)$; it becomes a scalar multiple of $Q_1(r+1:r+p,k)$.

After step $k$ of Algorithm~\ref{alg:alg1} is completed, we set $U_1 = [Q_1|_{\off}(:,1), Q(r+1:r+p,n+1)]$, as in~\eqref{eq:span_tridiag}. The matrix $V_1$ stores the coefficients in the basis $U_1$ of the columns in $Q_1|_{\off}$. Initially, $V_1 = [e_1, \mathbf{0}]$ with the first unit vector $e_1$. As we also need
to update the basis coefficients of $Q(r+1:r+p,n+1)$, we actually consider the augmented matrix   
$V^{*} = \big[V_1^{*} \quad {0 \atop 1}\big]$. In all subsequent steps of Algorithm~\ref{alg:alg1}, 
we only apply Givens rotations to the corresponding columns of $V^{*}$. Note that the last column of $V^{*}$ is only rescaled, as it is always combined with a zero column. 

After completing step $k+s$ of Algorithm~\ref{alg:alg1},
$Q_1|_{\off}$ remains unchanged and we extract the factor $V_1$ from the first $s$ columns of $V$.

Using the described procedure, the overall complexity to compute a low rank representation of $Q_1|_{\off}$ is 
$\mathcal{O}(\max\lbrace p,s \rbrace)$.
The off-diagonal blocks of $Q_2$ are treated analogously. 
\end{remark}

The QDWH algorithm makes use of the matrix product $Q_1 Q_2^*$, see~\eqref{eq:qdwh_qr}. Theorem~\ref{thm:ranks_tridiag}, together with the 
upper Hessenberg/triangular structure, directly implies that the ranks of the lower and upper off-diagonal blocks of $Q_1 Q_2^*$ are bounded by three and two, respectively. In fact, 
the following theorem shows a slightly stronger result.
\begin{theorem}
\label{thm:rank_tridiag_product}
Every off-diagonal block of $Q_1 Q_2^{*}$ for the orthogonal factor returned by Algorithm~\ref{alg:alg1} has rank at most $2$.  
\end{theorem}
\begin{proof}
By Algorithm~\ref{alg:alg1} and Theorem~\ref{thm:ranks_tridiag}, the matrices $Q_1$ and $Q_2$ admit for any $1 \leq k < n$ a partitioning of the from
\begin{equation*}
Q_1 = \left[ \begin{array}{@{\,} c|c @{\,}}
  X_1 & U_1V_1^{*}\\      
  \hline
  \sigma e_1e_k^{*} & X_2 \\
\end{array} \right], \quad
Q_2 = \left[ \begin{array}{@{\,} c|c @{\,}}
  Y_1 & U_2V_2^{*}\\      
  \hline
  \phantom{e_1}\mathbf{0}\phantom{e} & Y_2 \\
\end{array} \right],
\end{equation*}
where $X_1\in \R^{k\times k}, X_2\in\R^{n-k\times n-k}$ are upper Hessenberg, $Y_1 \in \R^{k\times k}, Y_2\in\R^{n-k\times n-k}$ are upper triangular, $U_1, U_2 \in \R^{k\times 2}$, $V_1, V_2\in\R^{n-k\times 2}$, $\sigma \in\R$, and 
$e_1, e_k$ denote unit vectors of appropriate lengths. The upper off-diagonal block of $Q_1Q_2^{*}$ equals to a product of rank-$2$ matrices
\begin{align*}
(Q_1 Q_2^{*})(1:k, k+1:n) &= X_1 \cdot \mathbf{0} + U_1 \underbrace{V_1^{*} Y_2^{*}}_{\widetilde{V}_1^{*}} = U_1\widetilde{V}_1^{*}\text{.}
\end{align*}
Moreover, the lower off-diagonal block amounts to a sum of a rank-$1$ and a rank-$2$ matrix
\begin{align*}
(Q_1 Q_2^{*})(k+1:n, 1:k) &= \sigma e_1 e_k^{*} Y_1^{*} + \underbrace{X_2 V_2}_{ \widetilde{V}_2 } U_2^{*}\\
&= \sigma e_1 Y_1(:,k)^{*} + \widetilde{V}_2 U_2^{*}\text{.}
\end{align*}  
If $\sigma = 0$, the statement holds. Otherwise, we first show that the vectors $Y_1(:, k)$ and $U_2(:,1)$ are collinear. Let us recall that the vectors $Y_1(:,k)$, $U_2(:,1)$ coincide with the vectors
 $Q(n+1:n+k+1,k)$, $Q(n+1:n+k+1,k+1)$ computed during step $k$ of Algorithm~\ref{alg:alg1}. As $Q(n+1:n+k+1,k)$ and $Q(n+1:n+k+1,k+1)$ are collinear
after performing step $k$, the same holds for $Y_1(:,k)$, $U_2(:,1)$, and $Y_1(:,k) = \eta U_2(:,1)$ for some $\eta \in \mathbb{R}$. Hence, we obtain
\begin{align*}
(Q_1 Q_2^{*})(k+1:n, 1:k) &= \sigma \eta e_1 U_2(:,1)^{*} + \widetilde{V}_2 U_2^{*} = \widehat{V}_2 U_2^{*}\text{,}
\end{align*} 
which completes the proof. 
\end{proof}

From Theorem~\ref{thm:rank_tridiag_product} and the recurrence~\eqref{eq:qdwh_qr_reccur} it follows that
the first iterate of the QDWH algorithm can be exactly represented in the HODLR format with off-diagonal ranks at most $3$. 

\subsection[QR decomposition for banded A]{QR decomposition of $\big[ {A \atop I} \big]$ for banded $A$}

In this section, we discuss the QR decomposition of $\big[ {A \atop I} \big]$ for a banded symmetric matrix $A$ with bandwidth $b> 1$. Let us first note that
Eld{\'e}n~\cite{Elden1984} proposed a fast algorithm for reducing a matrix $\big[ {A \atop L} \big]$ with an upper triangular banded matrix $L$. 
Eld{\'e}n's algorithm does not cover the fast computation of the orthogonal factor and requires the application of $(2b+1)n + nb - \frac{5}{2}b^2 - \frac{3}{2}b$ Givens rotations. In the following, we propose a different algorithm that only requires $(2b+1)n - b^2 - b$ Givens rotations.
 
Figure~\ref{fig:gap_6_band_res} illustrates the idea of our algorithm for $n = 6$ and $b = 3$. In the $i$th step of the algorithm, the subdiagonal elements in the $i$th column of $\big[ {A \atop I} \big]$ are annihilated as follows. A first group of Givens rotations ($\alpha_{i,j}$) annihilates all elements in row $n+i$, which consists of the diagonal element of $I$ and fill-in from the previous step.  Then a Givens rotation ($\beta_{i}$) annihilates the element $(n+1,i)$. Finally, a second group of Givens rotations ($\gamma_{i,j}$) annihilates all subdiagonal elements of $A$.
The detailed procedure is given in Algorithm~\ref{alg:alg2}. 

\begin{figure}[!ht]
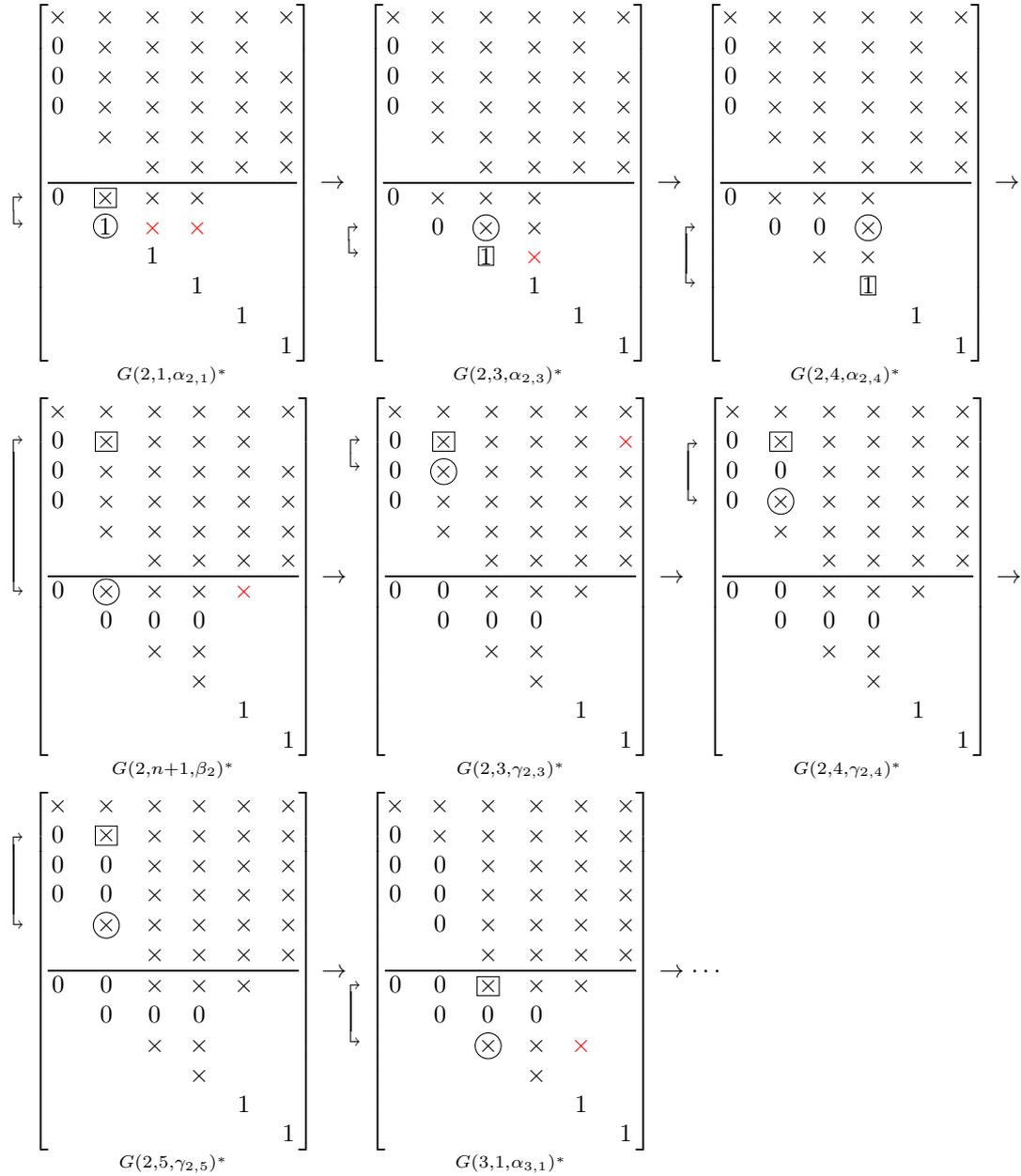

\begin{align*}
&\begin{array}{c@{\hspace{1mm}}}
        \\
        \\
        \\
        \\
        \\
  \Rc   \\[-0.01cm]
  \rc   \\[-0.4cm]
        \\
        \\   
        \\
        \\
 \end{array}
\underset{G(2,1,\alpha_{2,1})^{*}}{\begin{lbmatrix}{6}
\times &\times &\times &\times &\times &\times\\
0 &\times &\times &\times &\times & \\
0 &\times &\times &\times  &\times &\times \\
0 &\times &\times &\times &\times  &\times\\
  &\times &\times &\times &\times &\times \\ 
  & &\times &\times &\times &\times \\ \hline
0  &\circledr{$\times$} &\times &\times\\
 &\circled{1} &\textcolor{red}{\times} &\textcolor{red}{\times}\\
 & &1 \\
 & & &1 \\
 & & & &1 \\
 & & & & &1 \\
\end{lbmatrix}} \rightarrow \hskip -7pt
\begin{array}{c@{\hspace{1mm}}}
        \\
        \\
        \\
        \\
        \\
        \\
  \Rc   \\[-0.01cm]
  \rc   \\[-0.4cm]
        \\   
        \\
        \\
 \end{array}
\underset{G(2,3,\alpha_{2,3})^{*}}{\begin{lbmatrix}{6}
\times &\times &\times &\times &\times &\times\\
0 &\times &\times &\times &\times & \\
0 &\times &\times &\times  &\times &\times \\
0 &\times &\times &\times &\times  &\times\\
  &\times &\times &\times &\times &\times \\ 
  & &\times &\times &\times &\times \\ \hline
0  &\times &\times &\times\\
 &0 &\circledt{$\times$} &\times\\
  & &\circledr{1}  &\textcolor{red}{\times}\\
 & & &1 \\
 & & & &1 \\
 & & & & &1 \\
\end{lbmatrix}} \rightarrow \hskip -7pt
\begin{array}{c@{\hspace{1mm}}}
        \\
        \\
        \\
        \\
        \\
        \\
  \Rc   \\[-0.01cm]
  \clr  \\
  \rc   \\[-0.4cm]
        \\
        \\   
 \end{array}
\underset{G(2,4,\alpha_{2,4})^{*}}{\begin{lbmatrix}{6}
\times &\times &\times &\times &\times &\times\\
0 &\times &\times &\times &\times & \\
0 &\times &\times &\times  &\times &\times \\
0 &\times &\times &\times &\times  &\times\\
  &\times &\times &\times &\times &\times \\ 
  & &\times &\times &\times &\times \\ \hline
0  &\times &\times &\times\\
 &0 &0 &\circledt{$\times$}\\
  & &\times &\times\\
 & & &\circledr{1}\\
 & & & &1 \\
 & & & & &1 \\
\end{lbmatrix}} \rightarrow \\
&\begin{array}{c@{\hspace{1mm}}}
   \Rc  \\[-0.35cm]
          \\
    \clr    \\
    \clr    \\
    \clr    \\
    \clr    \\
   \rc     \\
        \\
        \\
        \\
        \\   
 \end{array}
\underset{G(2,n+1,\beta_2)^{*}}{\begin{lbmatrix}{6}
\times &\times &\times &\times &\times &\times\\
0 &\circledr{$\times$} &\times &\times &\times & \\
0 &\times &\times &\times  &\times &\times \\
0 &\times &\times &\times &\times  &\times\\
  &\times &\times &\times &\times &\times \\ 
  & &\times &\times &\times &\times \\ \hline
0  &\circledt{$\times$} &\times &\times &\textcolor{red}{\times}\\
 &0 &0 &0\\
  & &\times &\times\\
 & & &\times\\
 & & & &1 \\
 & & & & &1 \\
\end{lbmatrix}} \rightarrow \hskip -7pt
\begin{array}{c@{\hspace{1mm}}}
    \Rc  \\[0.0cm]
    \rc  \\[-0.35cm]
          \\
           \\
           \\
          \\
          \\
        \\
        \\
        \\
        \\   
 \end{array}
\underset{G(2,3,\gamma_{2,3})^{*}}{\begin{lbmatrix}{6}
\times &\times &\times &\times &\times &\times\\
0 &\circledr{$\times$} &\times &\times &\times &\textcolor{red}{\times} \\
0 &\circledt{$\times$} &\times &\times  &\times &\times \\
0 &\times &\times &\times &\times  &\times\\
  &\times &\times &\times &\times &\times \\ 
  & &\times &\times &\times &\times \\ \hline
0  &0 &\times &\times &\times\\
 &0 &0 &0\\
  & &\times &\times\\
 & & &\times\\
 & & & &1 \\
 & & & & &1 \\
\end{lbmatrix}} \rightarrow \hskip -7pt
\begin{array}{c@{\hspace{1mm}}}
      \Rc  \\[0.0cm]
      \clr  \\
      \rc  \\[-0.45cm]
          \\
           \\
          \\
          \\
        \\
        \\
        \\
        \\   
 \end{array}
\underset{G(2,4,\gamma_{2,4})^{*}}{\begin{lbmatrix}{6}
\times &\times &\times &\times &\times &\times\\
0 &\circledr{$\times$} &\times &\times &\times &\times\\
0 &0 &\times &\times  &\times &\times \\
0 &\circledt{$\times$} &\times &\times &\times  &\times\\
  &\times &\times &\times &\times &\times \\ 
  & &\times &\times &\times &\times \\ \hline
0  &0 &\times &\times &\times\\
 &0 &0 &0\\
  & &\times &\times\\
 & & &\times\\
 & & & &1 \\
 & & & & &1 \\
\end{lbmatrix}} \rightarrow \hskip -7pt\\
&\begin{array}{c@{\hspace{1mm}}}
      \Rc  \\[0.0cm]
      \clr  \\
      \clr  \\
      \rc  \\[-0.45cm]
          \\
          \\
          \\
          \\
          \\
          \\
          \\   
 \end{array}
\underset{G(2,5,\gamma_{2,5})^{*}}{\begin{lbmatrix}{6}
\times &\times &\times &\times &\times &\times\\
0 &\circledr{$\times$} &\times &\times &\times &\times\\
0 &0 &\times &\times  &\times &\times \\
0 &0 &\times &\times &\times  &\times\\
  &\circledt{$\times$} &\times &\times &\times &\times \\ 
  & &\times &\times &\times &\times \\ \hline
0  &0 &\times &\times &\times\\
 &0 &0 &0\\
  & &\times &\times\\
 & & &\times\\
 & & & &1 \\
 & & & & &1 \\
\end{lbmatrix}} \rightarrow \hskip -7pt
\begin{array}{c@{\hspace{1mm}}}
       \\
        \\
        \\
        \\
        \\
  \Rc   \\[-0.01cm]
   \clr \\
  \rc   \\[-0.4cm]
        \\
        \\   
        \\  
 \end{array}
\underset{G(3,1,\alpha_{3,1})^{*}}{\begin{lbmatrix}{6}
\times &\times &\times &\times &\times &\times\\
0 &\times &\times &\times &\times &\times\\
0 &0 &\times &\times  &\times &\times \\
0 &0 &\times &\times &\times  &\times\\
  &0 &\times &\times &\times &\times \\ 
  & &\times &\times &\times &\times \\ \hline
0  &0 &\circledr{$\times$} &\times &\times\\
 &0 &0 &0\\
  & &\circledt{$\times$} &\times &\textcolor{red}{\times}\\
 & & &\times\\
 & & & &1 \\
 & & & & &1 \\
\end{lbmatrix}} \rightarrow  \cdots
\end{align*}
\caption{Second step of fast QR decomposition (Algorithm~\ref{alg:alg2}) of $\big[ {A \atop I} \big]$ for banded $A$ with $n = 6$ and $b = 3$.
In each step, a Givens rotation is applied to the rows denoted by the arrows. Crosses denote generically nonzero elements, boxed/circled crosses are used to define Givens rotations, while 
red crosses denote the fill-in during the current operation.}
\label{fig:qr_example_banded}
\end{figure}

\begin{algorithm}[ht!]
    \caption{\text{Fast QR decomposition~\eqref{eq:qrdecomp} for banded $A$}}
    \label{alg:alg2}
    \renewcommand{\algorithmicrequire}{\textbf{Input:}}
    \renewcommand{\algorithmicensure}{\textbf{Output:}}

   \begin{algorithmic}[1]
    \REQUIRE Banded matrix $A$ with bandwidth $b$.
    \ENSURE Factors $Q,R$ of a QR decomposition of $\big[ {A \atop I} \big]$. 
   
   \STATE $Q \gets I_{2n}$, $R \gets \big[ {A \atop I} \big]$.

   \STATE Construct $G(1,n+1,\beta_1)$ to annihilate $R(n+1,1)$.
   \STATE Update $R \gets G(1,n+1,\beta_1)^{*}R$ and $Q \gets QG(1,n+1,\beta_1)$
   
    \FOR{$j = 2,\ldots,b+1$} 
     \STATE Construct $G(1,j,\gamma_{1,j})$ to annihilate $R(j,1)$.
     \STATE Update $R \gets  G(1,j,\gamma_{1,j})^{*}R$ and $Q \gets QG(1,j,\gamma_{1,j})$.
    \ENDFOR
      
   \vskip 1pt
   
     \FOR{$i = 2,\ldots,n$} 
      
        \STATE Construct $G(n+1,n+i,\alpha_{i,i})$ to annihilate $R(n+i,i)$.  
        \STATE Update $R \gets G(n+1,n+i,\alpha_{i,i})^{*}R$ and $Q \gets QG(n+1,n+i,\alpha_{i,i})$. 
       
       \vskip 1pt

       \FOR{$j = i+1,\ldots,\min \lbrace n,b+i-1\rbrace$} 
         \STATE Construct $G(n+i,n+j,\alpha_{i,j})$ to annihilate $R(n+i,j)$.  
        \STATE Update $R \gets G(n+i,n+j,\alpha_{i,j})^{*}R$ and $Q \gets QG(n+i,n+j,\alpha_{i,j})$.
       \ENDFOR

       \vskip 1pt
       \STATE Construct $G(i,n+1,\beta_i)$ to annihilate  $R(n+1,i)$.
       \STATE Update $R \gets G(i,n+1,\beta_i)^{*}R$ and $Q \gets QG(i,n+1,\beta_i)$. 
       
       \vskip 1pt
      
      \IF{$i<n$}
       \FOR{$j = i+1,\ldots,\min\lbrace n,b+i\rbrace$} 
          \STATE Construct $G(i,j,\gamma_{i,j})$ to annihilate $R(j,i)$.
          \STATE Update $R \gets G(i,j,\gamma_{i,j})^{*}R$ and $Q \gets QG(i,j,\gamma_{i,j})$.
          
       \ENDFOR
     \ENDIF
  \ENDFOR

 \end{algorithmic}
\end{algorithm}

\subsubsection{Ranks of off-diagonal blocks and fast computation of orthogonal factor}

Due to the order of annihilation in Algorithm~\ref{alg:alg2}, it follows that $Q_1 = Q(1:n,1:n)$ is a $b$-Hessenberg matrix (that is, the matrix is zero below the $b$th subdiagonal) while
$Q_2 = Q(n+1:2n,1:n)$ is an upper triangular matrix. 
The following result and its proof yield an $\calO(b^2 n)$ algorithm for computing $Q_1$ and $Q_2$, analogous to Theorem~\ref{thm:ranks_tridiag} and Remark~\ref{construct_off}. 

\begin{theorem}
\label{thm:ranks_banded}
For the orthogonal factor $Q$ returned by Algorithm~\ref{alg:alg2}, it holds that the matrices $Q(1:k,k+1:n)$ and $Q(n+1:n+k,k+1:n)$ have rank at most $2b$ for all $1 \leq k < n$.
\end{theorem}

\begin{proof}
Again, we prove the result for $Q(1:k,k+1:n)$ only.
After $k$ steps of Algorithm~\ref{alg:alg2} have been performed, we define the subspace
\begin{align*}
 \mathcal{U} := \myspan \lbrace & Q(1:k,k+1),\ldots,Q(1:k,k+b),Q(1:k,n+1),\\
 & Q(1:k,n+k+1),\ldots, Q(1:k,n+k+b-1) \rbrace,
\end{align*}
which is of dimension not larger than $2b$. At this point, the columns $Q(1:k,j)$ are zero for $j = k+b+1,\ldots,n$ and $j = n+k+b,\ldots,2n$. Thus,
\begin{equation} \label{eq:relationbanded}
\myspan Q(1:k,k+1:n+1) \subset \calU, \qquad \myspan Q(1:k, n+k+1:2n) \subset \calU
\end{equation}
hold after $k$ steps of Algorithm~\ref{alg:alg2}. We now show by induction that this relation holds for all subsequent steps.

Suppose that~\eqref{eq:relationbanded} holds after $i$ steps with $k\le i\le n-1$. In step $i+1$, the following 
operations are performed by Algorithm~\ref{alg:alg2}:

\begin{enumerate}
\item $G(n+1,n+i+1,\alpha_{i+1,i+1})$ is applied to columns $n+1$ and $n+i+1$ of $Q$, which affects and preserves both inclusions in~\eqref{eq:relationbanded}. Then 
$G(n+i+1,n+j,\alpha_{i+1,j})$ is applied to columns $n+i+1$ and $n+j$ of $Q$, for $j=i+2:\min \lbrace n,i+b\rbrace$, hence 
 $\myspan Q(1:k, n+k+1:2n) \subset \calU$ remains true.     
\item  $G(i+1,n+1,\beta_{i+1})$ is applied to columns $i+1$ and $n+1$ of $Q$, preserving  $\myspan Q(1:k,k+1:n+1) \subset \calU$.
\item If $i+1< n$, $G(i+1,j,\gamma_{i+1,j})$ is applied to columns $i+1$ and $j$ of $Q$, for $j = i+2:\min \lbrace n, i+b+1\rbrace$, which 
retains $\myspan Q(1:k,k+1:n+1) \subset \calU$.
\end{enumerate}
Therefore~\eqref{eq:relationbanded} holds after Algorithm~\ref{alg:alg2} has been completed, which completes the proof of the theorem.
\end{proof}

The following result is an extension of Theorem~\ref{thm:rank_band_product} from the tridiagonal to the banded case. Its proof is very similar and therefore omitted.

\begin{theorem}
\label{thm:rank_band_product}
Every off-diagonal block of $Q_1Q_2^{*}$ for the orthogonal factor returned by Algorithm~\ref{alg:alg2} has rank at most $2b$. 
\end{theorem}

\section{hQDWH algorithm}
\label{overall_algorithm}

Algorithm~\ref{alg:alg3}, summarizes the hQDWH algorithm proposed in this paper.

\begin{algorithm}[h!]
    \caption{\text{hQDWH algorithm}}
    \label{alg:alg3}
    \renewcommand{\algorithmicrequire}{\textbf{Input:}}
    \renewcommand{\algorithmicensure}{\textbf{Output:}}

   \begin{algorithmic}[1]
     
   \REQUIRE Symmetric banded matrix $A$ with bandwidth $b\ge 1$, minimal block-size $n_{\min} \ge 2$, truncation tolerance $\epsilon>0$, stopping tolerance $\delta>0$.
   \ENSURE Approximation $P$ in HODLR format to spectral projector $\Pi_{<0}(A)$. 
     
   \STATE Choose  initial parameters $\alpha, l_0$ of QDWH according to~\eqref{eqref:param_estimation}.\label{alg_alpha_l}
   \STATE $X_0 = A/\alpha$. 
   \STATE $k = 0$.

   \WHILE{$\vert 1 - l_k\vert > \delta$} \label{alg_stopping}   
     \STATE Compute $a_k$, $b_k$, $c_k$ according to the recurrence~\eqref{eq:qdwh_parameters_abc}.
      \IF{$k = 0$}
        \STATE Apply $\begin{displaystyle}\begin{cases} 
                    \text{Algorithm~\ref{alg:alg1},}  &\text{for } b = 1\\
                    \text{Algorithm~\ref{alg:alg2},}  &\text{for } b > 1\\
                      \end{cases}\end{displaystyle}$
               to $\begin{bmatrix}\sqrt{c_0}X_0\\ I\end{bmatrix}$ and store resulting array $G$ of Givens rotations.\label{alg_compute_givens} 
        \STATE Compute $Q_1$ and $Q_2$ from $G$ in HODLR format; see Remark~\ref{construct_off}. \label{alg_compute_Q} 
        \STATE $X_{1} = \frac{b_0}{c_0}*_{\h}X_0 +_{\h}\frac{1}{\sqrt{c_0}}\left(a_0 - \frac{b_0}{c_0}\right)*_{\h}Q_1*_{\h}Q_2^{*}$. \label{alg_QR_it}
      \ELSE
         \STATE $W_k = \h\operatorname{-Cholesky}(I +_{\h} c_k*_{\h}X_k^{*}*_{\h}X_k)$. \label{alg_chol1} 
         \STATE Solve upper-triangular system $Y_kW_k = X_k$ in HODLR format. 
         \STATE Solve lower-triangular system $V_kW_k^{*} = Y_k$ in HODLR format. 
        \STATE  $X_{k+1} = \frac{b_k}{c_k}*_{\h}X_k +_{\h} \left(a_k - \frac{b_k}{c_k}\right)*_{\h}V_k$. \label{alg_chol}  
     \ENDIF
     \STATE $k = k+1$.
     \STATE $l_k = l_{k-1}(a_{k-1} + b_{k-1}l^{2}_{k-1})/(1 + c_{k-1}l^{2}_{k-1})$. 
   \ENDWHILE
   \STATE $U = X_k$. \label{alg_sign} 
   \STATE $P =\frac{1}{2}*_{\h}(I -_{\h}U)$. \label{alg_sp} 
   \end{algorithmic}
\end{algorithm}

\noindent In the following, we comment on various implementation details of Algorithm~\ref{alg:alg3}.
\begin{description}
\item[line~\ref{alg_alpha_l}] As proposed in~\cite{NakaHigh2013}, the parameters $\alpha \gtrsim \Vert A \Vert_2$ and $l_0 \lesssim \sigma_{\min}(X_0)$ needed to start the QDWH algorithm are estimated as
\begin{equation}
\label{eqref:param_estimation}
 \alpha = \texttt{normest}(A), \quad l_0 = \Vert A/\alpha \Vert_1 / (\sqrt{n}\cdot \texttt{condest}(A/\alpha))\text{,}
\end{equation}
where \texttt{normest} and \texttt{condest} denote the \Matlab{} functions for estimating the matrix $2$--norm using the power method and the $1$--norm condition number using~\cite{HighTiss2000}, respectively. Both functions exploit that $A$ is sparse and require $\calO(bn)$ and $\calO(b^2 n)$ operations, respectively.

\item[lines~\ref{alg_compute_givens}--~\ref{alg_QR_it}] This part of the algorithm deals with the implementation of the first QR-based iterate~\eqref{eq:qdwh_qr_reccur}. The generation of Givens rotations by Algorithms~\ref{alg:alg1} and~\ref{alg:alg2} for reducing $\big[ {\sqrt{c_0}X_0 \atop I} \big]$ to triangular form has been implemented in a \textsc{C} function, making use of the LAPACK routine {\tt DLARTG}. The function is called via a MEX interface and returns an array $G$ containing the cosines and sines of all rotations. This array is then used 
in~\ref{alg_compute_Q} to generate $Q_1$ and $Q_2$ in the HODLR format, whose precise form is defined by the input parameter $n_{\min}$.
\item[lines~\ref{alg_chol1}--~\ref{alg_chol}] The computation of the $k$th iterate $X_k$, $k>1$, involves the Cholesky decomposition, addition, and the solution of triangular linear systems in the HODLR format. Existing techniques for HODLR matrices have been used for this purpose, see Section~\ref{hmatrices}, and repeated recompression with the absolute truncation tolerance $\epsilon$ is applied.
 \end{description}

\begin{remark}
Algorithm~\ref{alg:alg3} extends in a straightforward way to the more general hierarchical matrix format from Section~\ref{sec:hierarchicalmatrices}. 
The only major difference is the need for converting the matrices after line~\ref{alg_QR_it} from the HODLR to the hierarchical matrix format.
This extension of Algorithm~\ref{alg:alg3} was used in Example~\ref{ex:hodlrvshmatrix}. 
\end{remark}

Assuming that all ranks in the off-diagonal blocks are bounded by $k \ge b$, Algorithm~\ref{alg:alg3} requires $\calO (kn \log n)$ memory and $\mathcal{O}(k^2n\log^2 n)$ operations.

\section{Numerical experiments}
\label{experiments}

In this section, we demonstrate the performance of our preliminary \Matlab{} implementation of the hQDWH algorithm. All 
computations were performed in \Matlab{} version 2014a on an Intel Xeon CPU with 3.07GHz, $4096$ KByte of level $2$ cache 
and $192$ GByte of RAM. To be able to draw a fair comparison, all experiments were performed on a single core. 

To measure the accuracy of the QDWH algorithm, we use the functions $e^{Q}_{\id}$, $e^{Q}_{\trace}$, $e^{Q}_{\SP}$ defined in~\eqref{eq:error_measures}. The error measures $e^{\h}_{\id}$,  $e^{\h}_{\trace}$, $e^{\h}_{\SP}$ for the hQDWH algorithm are defined analogously. 
In all experiments, we used the tolerance $\delta = 10^{-15}$ for stopping the QDWH/hQDWH algorithms. Unless stated otherwise, the truncation tolerance for recompression in the HODLR format is set to 
$\epsilon = 10^{-10}$; the minimal block-size is set to 
$n_{\min} = 250$ for tridiagonal matrices and $n_{\min} = 500$ for banded matrices. 

The performance of the algorithm is tested on various types of matrices, including synthetic examples as well as examples from widely used sparse matrix collections.
  
\subsection{Construction of synthetic test matrices}
\label{sec:construct_test_matrices}

Given a prescribed set of eigenvalues $\lambda_1,\ldots, \lambda_n$ and a bandwidth $b$, we construct a symmetric $b$--banded matrix 
by an orthogonal similarity transformation of $A = \diag(\lambda_1,\ldots, \lambda_n)$.
For this purpose, we perform the following operation for $i = n,n-1,\ldots,2$:

First, a Givens rotations $G(i-1,i,\alpha_i)$ is created by annihilating the second component of the vector $\big[ {a_{ii} \atop 1 } \big]$. The update $A \gets G(i-1,i,\alpha_i)^* A G(i-1,i,\alpha_i)$ introduces nonzero off-diagonal elements in $A$. For 
$i=n,\ldots,n-b+1$, this fill-in stays within the $b$ bands. For $i \le n-b$, two undesired nonzero elements are created in row $i-1$ and column $i-1$ outside the $b$ bands. These nonzero elements are immediately chased off to the bottom right corner by applying $n-b-i+1$ Givens rotations, akin to Schwarz band reduction~\cite{Schwarz1968}.

When the procedure is completed, the $b$ bands of $A$ are fully populated.

In all examples below, we choose the eigenvalues to be uniformly distributed in $[-1,-\gap] \cup [\gap,1]$. Our results indicate that the 
performance of our algorithm is robust with respect to the choice of eigenvalue distribution. In particular, the timings stay 
almost the same when choosing a distribution geometrically graded towards the spectral gap. 

\subsection{Results for tridiagonal matrices}

\begin{example}[\bfseries{Accuracy versus $\gap$}]
\label{ex:error_vs_gap_tridiag} 
\rm First we investigate the behavior of the errors for hQDWH and QDWH with respect to the spectral gap. Using the construction from Section~\ref{sec:construct_test_matrices}, we consider $10000 \times 10000$ tridiagonal matrices with
eigenvalues in $[-1,\hskip 3pt -\gap] \cup [\gap, \hskip 3pt 1]$, where 
$\gap$ varies from $10^{-15}$ to $10^{-1}$.  From Figure~\ref{fig:error_vs_gap_tol_tridiag} (left), it can be seen that 
tiny spectral gaps do not have a significant influence on the distance from identity and the trace error for both algorithms. 
On the other hand, both $e^{\h}_{\SP}$ and $e^{Q}_{\SP}$ are sensitive to a decreasing gap, which reflects the ill-conditioning of the spectral projector for small gaps.   
\end{example}

\begin{figure}[!h]
\begin{center}
\includegraphics[width=0.48\textwidth]{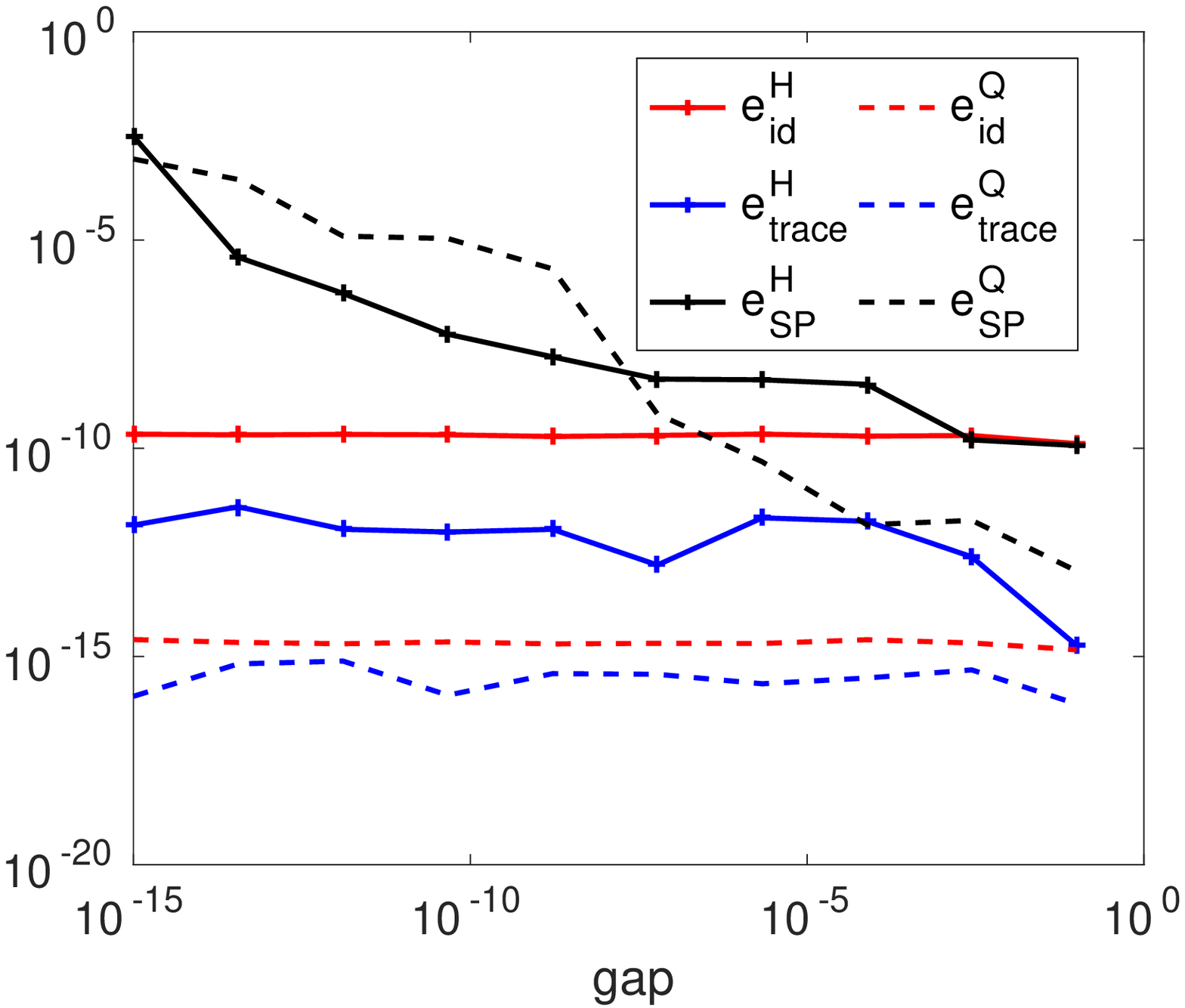}
\includegraphics[width=0.48\textwidth]{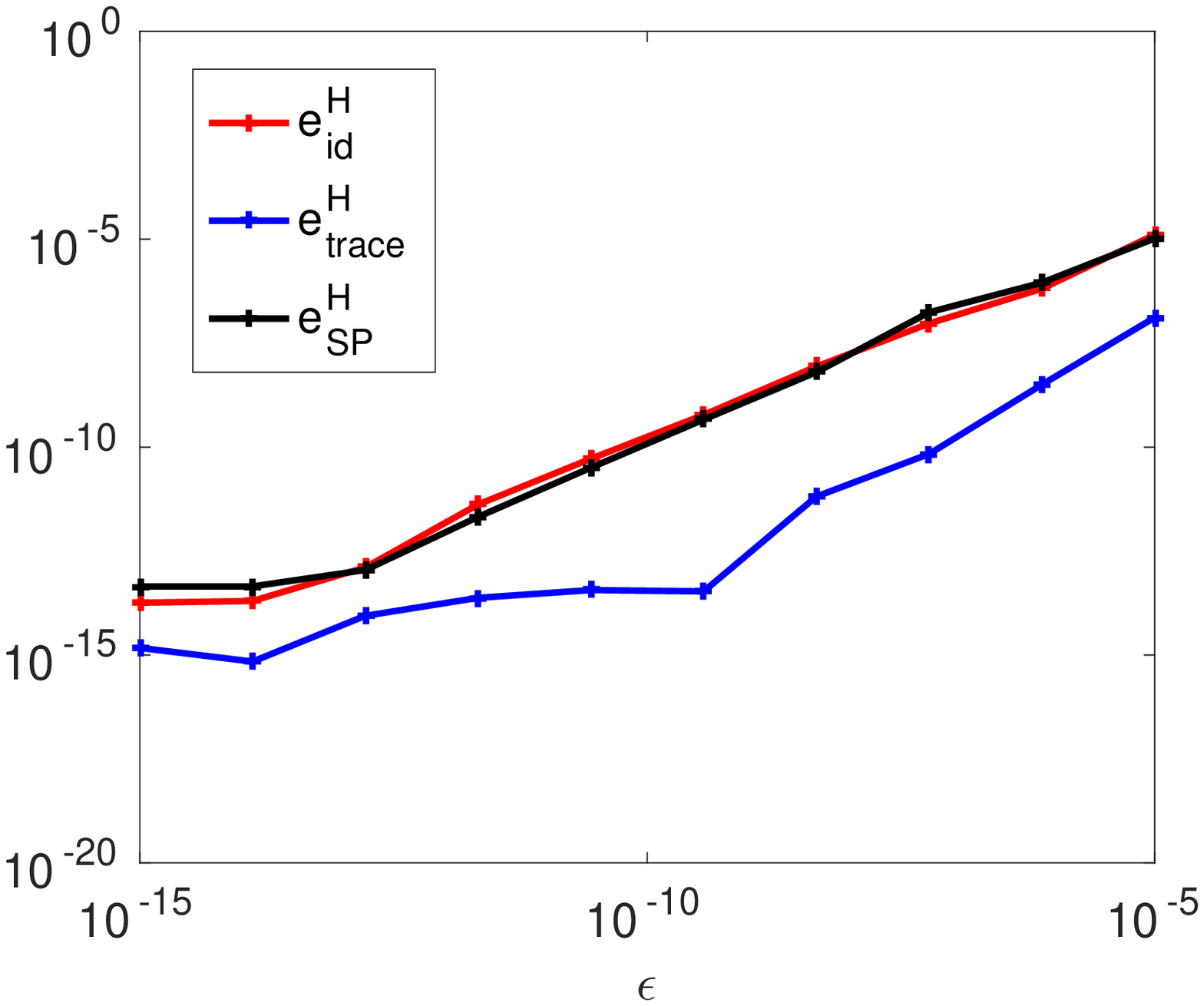}
\end{center}
\caption{Left (Example~\ref{ex:error_vs_gap_tridiag}): Comparison of accuracy for hQDWH and QDWH applied to tridiagonal matrices. Right (Example~\ref{ex:error_vs_tolerance_tridiag}): 
Accuracy of hQDWH for different truncation tolerances.}
\label{fig:error_vs_gap_tol_tridiag}
\end{figure}

\begin{example}[\bfseries{Accuracy versus $\epsilon$}]
 \label{ex:error_vs_tolerance_tridiag} 
{\rm We again consider a tridiagonal matrix $A\in\mathbb{R}^{10000\times 10000}$, with the eigenvalues
in $[-1, \hskip 3pt -10^{-4}] \cup [10^{-4}, \hskip 3pt 1]$. The truncation tolerance $\epsilon$ for recompression in the HODLR format is varied in the interval
$[10^{-15}, 10^{-5}]$. Figure~\ref{fig:error_vs_gap_tol_tridiag} shows the resulting errors in the hQDWH algorithm. As
 expected, the errors  $e^{\h}_{\id}$,  $e^{\h}_{\trace}$, $e^{\h}_{\SP}$ increase as $\epsilon$ increases. Both $e^{\h}_{\id}$ and $e^{\h}_{\SP}$ grow linearly with 
 respect to $\epsilon$, while $e^{\h}_{\trace}$ appears to be a little more robust. }
\end{example}

\begin{example}[\bfseries{Accuracy for examples from matrix collections}]
\label{ex:accuracy_applications}
\textnormal{We tested the accuracy of the hQDWH algorithm for the matrices from applications also considered in~\cite{MarqVomeDemmParl2009}:}
\begin{itemize}
\item \textnormal{Matrices from the BCSSTRUC1 set in the Harwell-Boeing Collection~\cite{Davis2007}. In these examples, a finite element discretization leads to a generalized eigenvalue problem $Kx = \lambda Mx$ with $K,M$ symmetric positive definite. 
We consider the equivalent standard eigenvalue problem $L^{-1}KL^{-T}$, where $L$ denotes the Cholesky factor of $M$, and shift the matrix such that 
 approximately half of its spectrum is negative. Finally, the matrix is reduced to a tridiagonal matrix using the Matlab{} function \texttt{hess}.}
\item \textnormal{Matrices from UF Sparse Matrix Collection~\cite{Davis2007}. We consider the symmetric Alemdar and Cannizzo matrices, as well as
  a matrix from the NASA set. Again, the matrices are shifted such that roughly half of their spectrum is negative and then reduced to tridiagonal form.}
\end{itemize}%
\begin{table}[h!]
\centering
{\renewcommand{\arraystretch}{1.2}
\begin{tabular}{c||c||c|c|c|c|c|c}

&matrix &$n$ &$e^{\h}_{\id}$ &$e^{\h}_{\trace}$ &$e^{\h}_{\SP}$ &$\Vert \cdot \Vert_2$  &$\gap$\\
 \hline
 \hline
 \multirow{3}{*}{\begin{sideways}\tiny{BCSSTRUC1}\end{sideways}} &\text{bcsst08} &$1074$ &$10^{-11}$ &$10^{-13}$ &$10^{-9}$ &$1.68\cdot10^{7}$ &$1.2615$\\
 \cline{2-8}
 &\text{bcsst09} &$1083$ &$10^{-10}$ &$10^{-11}$ &$10^{-8}$ &$4.29\cdot 10^{14}$ &$2.13\cdot 10^{10}$\\
 \cline{2-8}
 &\text{bcsst11} &$1474$ &$10^{-10}$ &$10^{-12}$ &$10^{-7}$ &$4.75\cdot 10^{9}$ &$1.58\cdot 10^{5}$\\
 \hline
 \hline
 &\text{Cannizzo matrix} &$4098$ &$10^{-10}$ &$10^{-10}$ &$10^{-7}$ &$3.07\cdot 10^{8}$ &$0.728$\\
  \cline{2-8}
 &\text{nasa4704} &$4704$ &$10^{-10}$ &$10^{-12}$ &$10^{-9}$ &$2.07\cdot 10^{8}$ &$20.182$\\
 \cline{2-8}
 &\text{Alemdar matrix} &$6245$ &$10^{-10}$ &$10^{-11}$ &$10^{-7}$ &$69.5$ &$0.0079$\\
 \cline{2-8} 
 \end{tabular}
 }
\caption{Accuracy of hQDWH for tridiagonal matrices from Example~\ref{ex:accuracy_applications}.}
 \label{table:error_test_matrices}
\end{table}%
{\rm Table~\ref{table:error_test_matrices} reveals that the hQDWH algorithm yields accurate approximations also
for applications' matrices. More specifically, in all examples,  $e^{\h}_{\id}$ and $e^{\h}_{\trace}$ obtain values of order of $\epsilon$, 
while $e^{\h}_{\SP}$ shows dependence on the relative spectral gap. } 
\end{example}

\begin{example}[\textbf{Breakeven point relative to {\tt eig}}]
\textnormal{To compare the computational times of the hQDWH algorithm with \texttt{eig}, we  
consider tridiagonal matrices with eigenvalues contained in $[-1, \hskip 3pt -\gap] \cup [\gap, \hskip 3pt 1]$ for various gaps. 
Table~\ref{table:break_even_point_tridiag} shows the resulting breakeven points, that is, the value of $n$ such that hQDWH is faster than {\tt eig} for matrices of size at least $n$. Not surprisingly,
this breakeven point depends on the gap, as ranks are expected to increase as the gap decreases.
However, even for $\gap = 10^{-4}$, the hQDWH 
algorithm becomes faster than \texttt{eig} for matrices of moderate size ($n \ge 3250$) and the ranks of the off-diagonal blocks in the HODLR representation of the spectral projector remain reasonably small.}   
\begin{table}[ht!]

\centering
{\renewcommand{\arraystretch}{1.2}
\begin{tabular}{c||c||c}
gap &breakeven point &max off-diagonal rank \\
 \hline
 \hline
 $10^{-1}$ &$n=2250$ &$18$\\
 \hline
 $10^{-2}$ &$n=2500$ &$28$\\
 \hline
 $10^{-3}$ &$n=2750$ &$35$ \\
 \hline
 $10^{-4}$ &$n=3250$ &$37$\\
 \hline
\end{tabular}
 }
 \caption{Breakeven point of hQDWH relative to \texttt{eig} for tridiagonal matrices. The last column shows the maximal off-diagonal rank in the output of hQDWH.}
 \label{table:break_even_point_tridiag}
 \end{table}
\end{example}

\begin{example}[\bfseries{Performance versus $n$}]
\label{ex:scaling_tridiag}
{\rm In this example, we investigate the asymptotic behavior of the hQDWH algorithm, in terms of computational time and memory, for tridiagonal matrices with eigenvalues in $[-1, \hskip 3pt -10^{-6}] \cup 
[10^{-6}, \hskip 3pt 1]$. Figure~\ref{fig:gap_6_tridiag_res} indicates that the expected $\mathcal{O}(n\log^2n)$ computational time and $\mathcal{O}(n\log n)$ are nicely matched. The faster increase for smaller $n$ is due to fact that the off-diagonal ranks first grow from $30$ to $64$ until they settle around $64$ for sufficiently large $n$. }

\begin{figure}[!h]
\begin{center}
\includegraphics[width=0.48\textwidth]{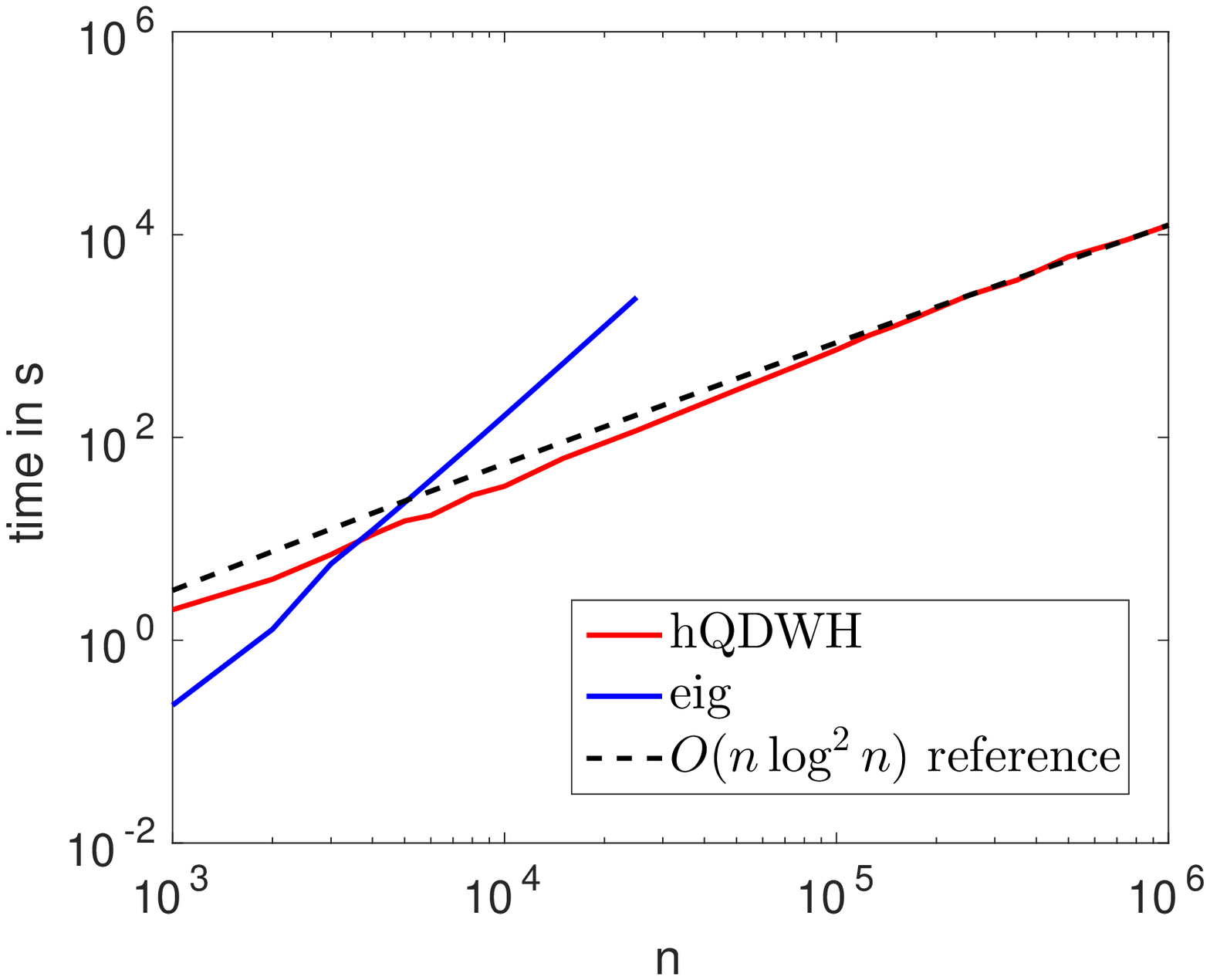}
\includegraphics[width=0.48\textwidth]{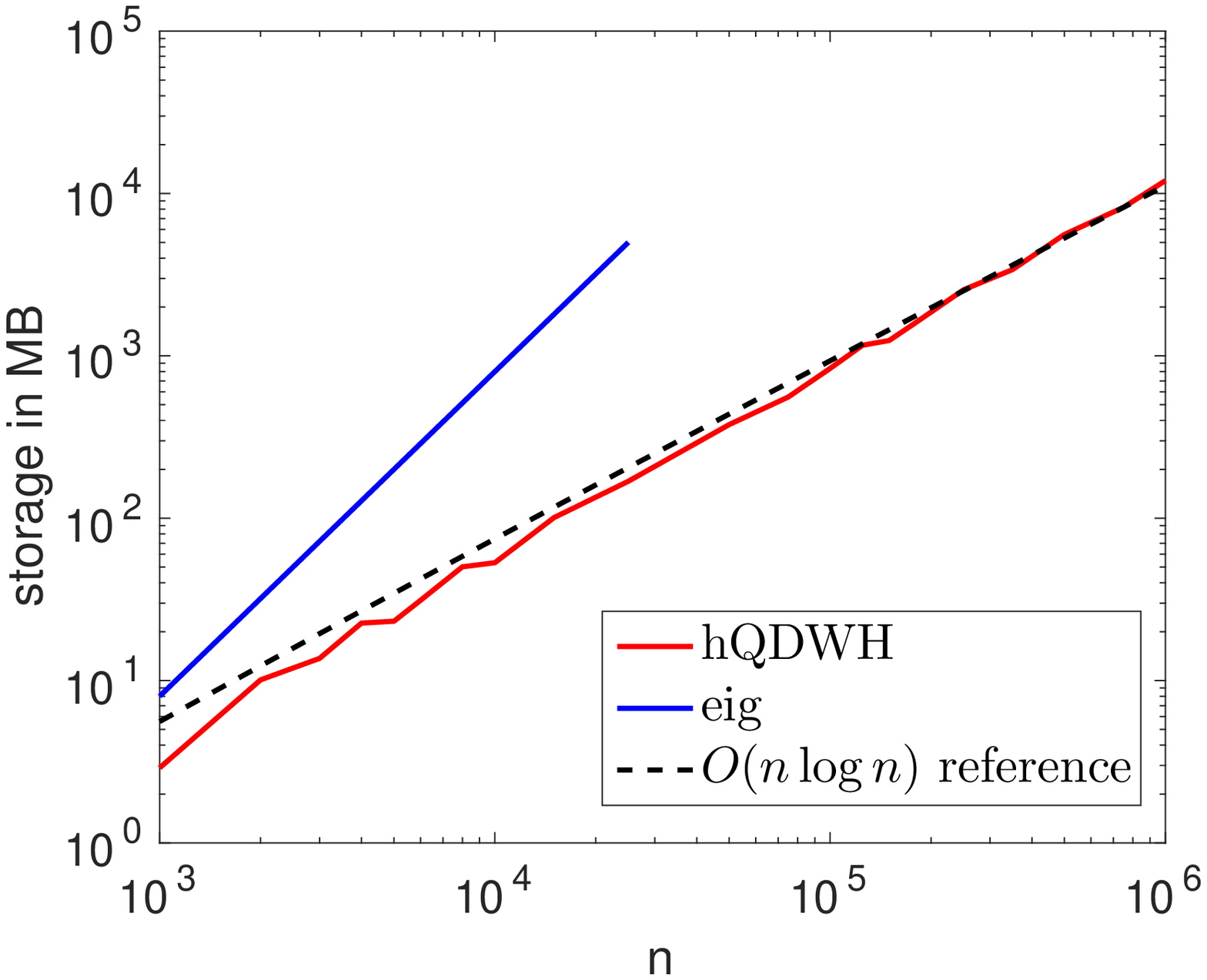}
\end{center}
\caption{Example~\ref{ex:scaling_tridiag}. Performance of hQDWH and \texttt{eig} applied to tridiagonal matrices with respect to $n$. Left: Computational time. Right: Memory requirements.}
\label{fig:gap_6_tridiag_res}
\end{figure}

\end{example}

\begin{example}[\bfseries{Performance for 1D Laplace}]
\label{ex:scaling_tridiag_gallery}
{\rm It is interesting to test the performance of the hQDWH algorithm for matrices for which the spectral gap decreases as $n$ increases. The archetypical example is the (scaled) tridiagonal matrix from the central difference discretization of the $1$D Laplace operator, with eigenvalues $\lambda_k = 2 - 2\cos\frac{k\pi}{n+1}$ for $k = 1,\ldots,n$. The matrix is shifted by $2$, such that half of its spectrum is negative and the eigenvalues become equal to $\lambda_k = - 2\cos\frac{k\pi}{n+1}$. The spectral gap is given by $\gap = 2\sin\frac{\pi}{n+1} = \calO(1/n^2)$. According to Theorem~\ref{sign_sing_value_decay}, the numerical ranks of the off-diagonal blocks depend logarithmically on the spectral gap. Thus, we expect that the hQDWH algorithm
requires $\mathcal{O}(n\log^4n)$ computational time and $\mathcal{O}(n\log^2 n)$ memory for this matrix. Figure~\ref{fig:gallery_tridiag} nicely confirms this expectation. }   

\begin{figure}[!h]
\begin{center}
\includegraphics[width=0.48\textwidth]{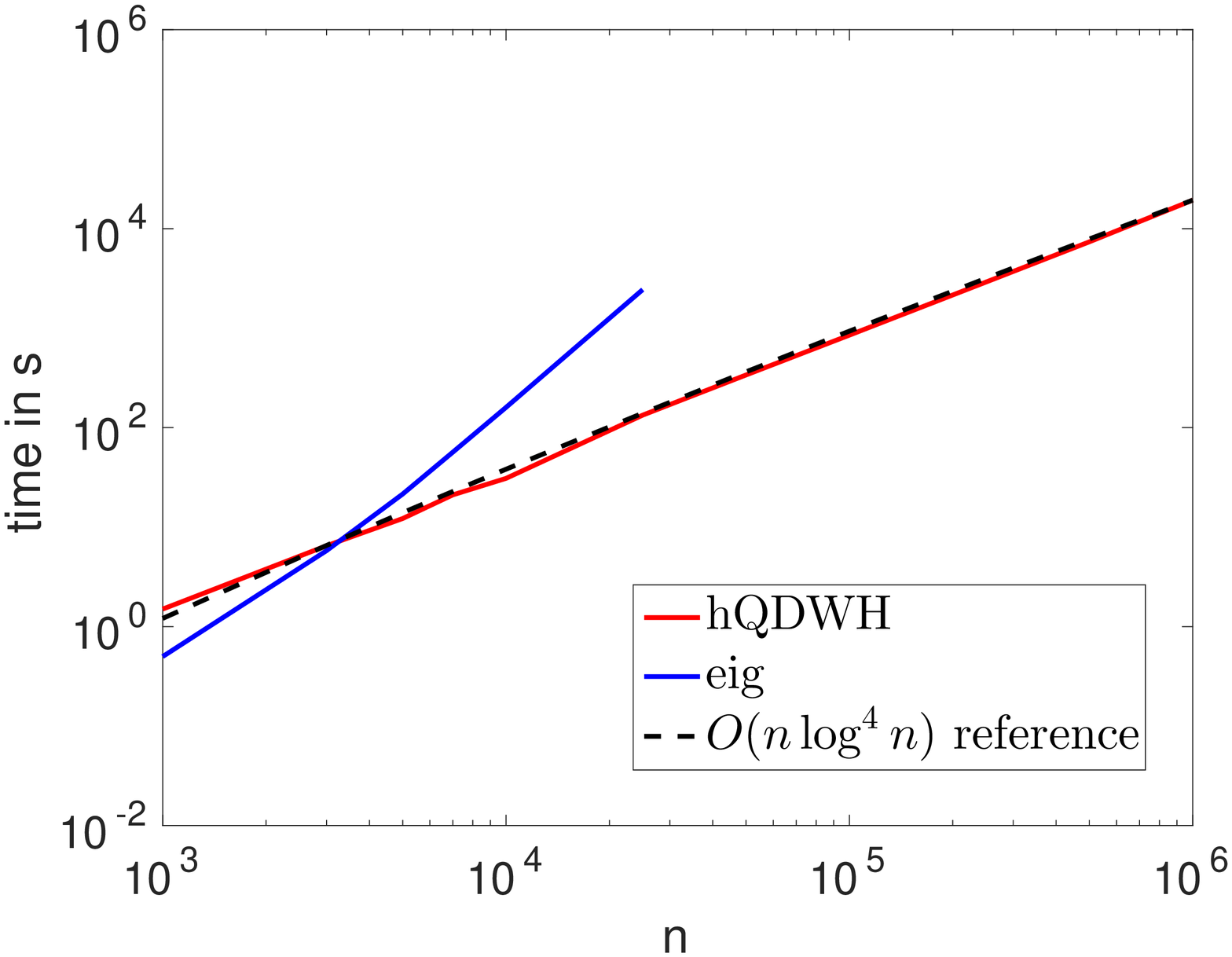}
\includegraphics[width=0.48\textwidth]{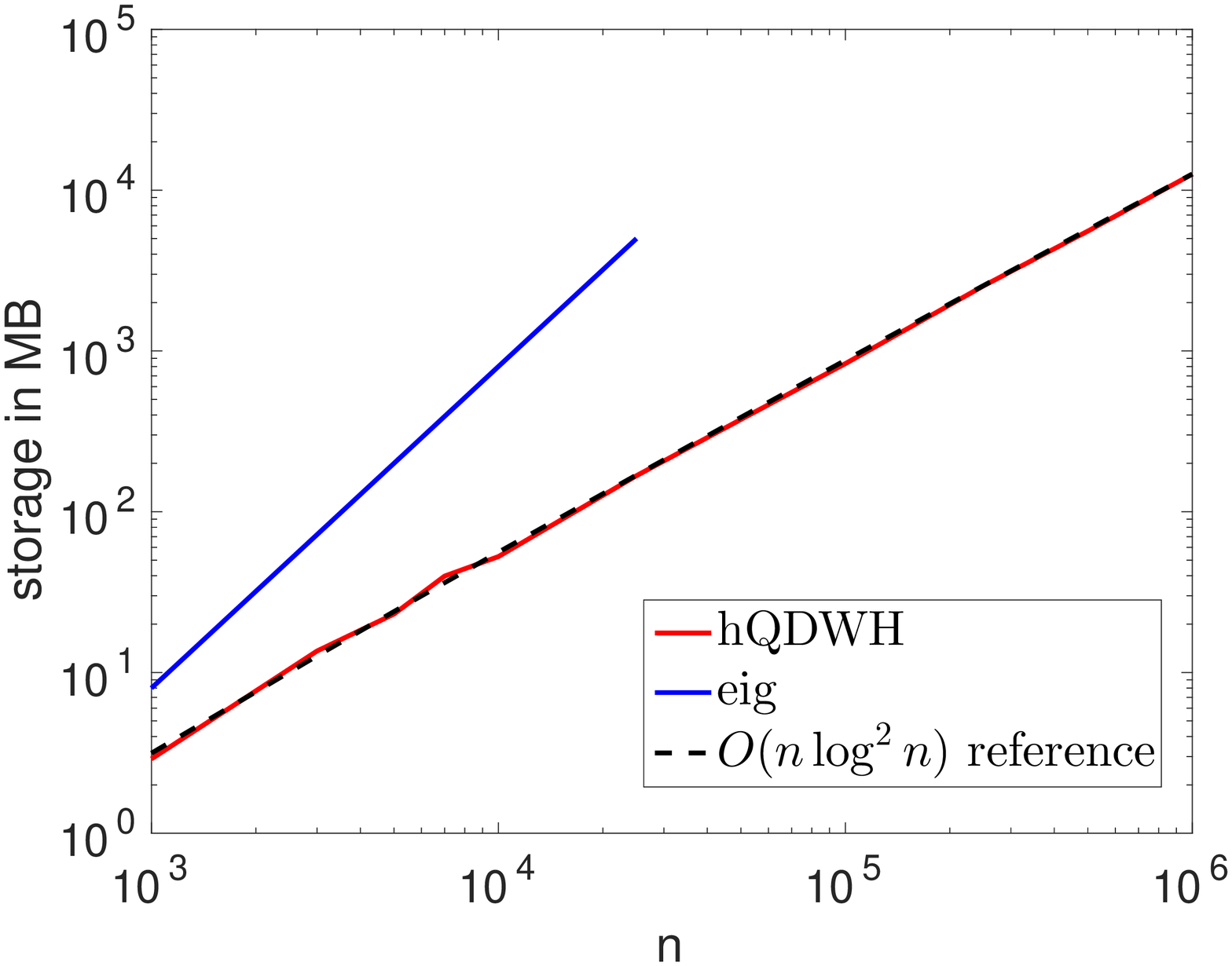}
\end{center}
\caption{Example~\ref{ex:scaling_tridiag_gallery}. hQDWH and \texttt{eig} for discretized 1D Laplace. Left: Computational time with respect to $n$. Right: Memory requirements with respect to $n$.}
\label{fig:gallery_tridiag}
\end{figure}
\end{example}

\begin{example}[\bfseries{Performance versus $n_{\min}$}]
 \label{ex:time_vs_nb_tridiag}  
{\rm The choice of the minimal block size $n_{\min}$ in the HODLR format influences the performance of hQDWH.
We have investigated this dependence for $50\,000\times 50\,000$ tridiagonal matrices with eigenvalues contained 
in $[-1, \hskip 3pt -\gap] \cup [\gap, \hskip 3pt 1]$, 
and $\gap \in \lbrace 10^{-1}, 10^{-4}, 10^{-6}\rbrace$. Figure~\ref{fig:time_vs_nb_tridiag} indicates that the optimal 
value of $n_{\min}$ increases for smaller gaps.  However, the execution time is not overly 
sensitive to this choice; a value of $n_{\min}$ between $200$ and $500$ leads to 
good performance.}
%

\begin{figure}[!h]
\begin{center}
\includegraphics[width=0.48\textwidth]{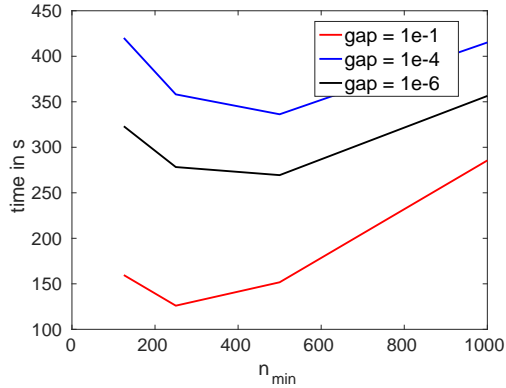}
\end{center}
\caption{Example~\ref{ex:time_vs_nb_tridiag}. Computational time of hQDWH versus $n_{\min}$.}
\label{fig:time_vs_nb_tridiag}
\end{figure}
\end{example}

\subsection{Results for banded matrices}

\begin{example}[\bfseries{Accuracy versus $\gap$}]
 \label{ex:band_error_vs_gap}
{\rm Similarly to Example~\ref{ex:error_vs_gap_tridiag}, we study the impact of the spectral gap on the accuracy of hQDWH and QDWH for banded matrices.
Using once again the construction from Section~\ref{sec:construct_test_matrices}, we consider $10000 \times 10000$ banded matrices with bandwidth $8$
and eigenvalues in $[-1,\hskip 3pt -\gap] \cup [\gap, \hskip 3pt 1]$, where 
$\gap$ varies from $10^{-15}$ to $10^{-1}$. The left plot of Figure~\ref{fig:error_vs_gap_tol_band} reconfirms the observations from Example~\ref{ex:error_vs_gap_tridiag}.}  
\end{example}

\begin{figure}[!h]
\begin{center}
\includegraphics[width=0.48\textwidth]{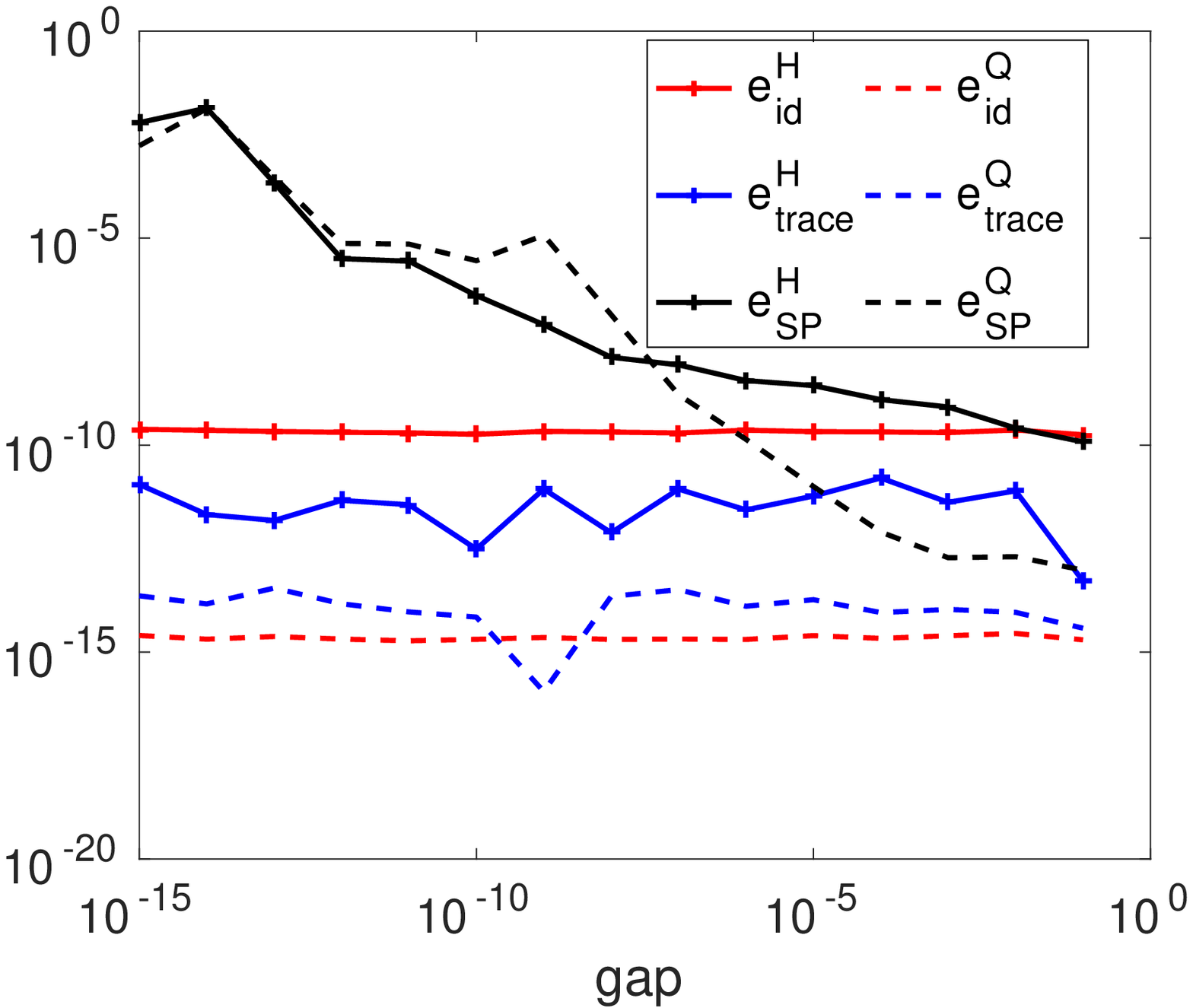}
\includegraphics[width=0.48\textwidth]{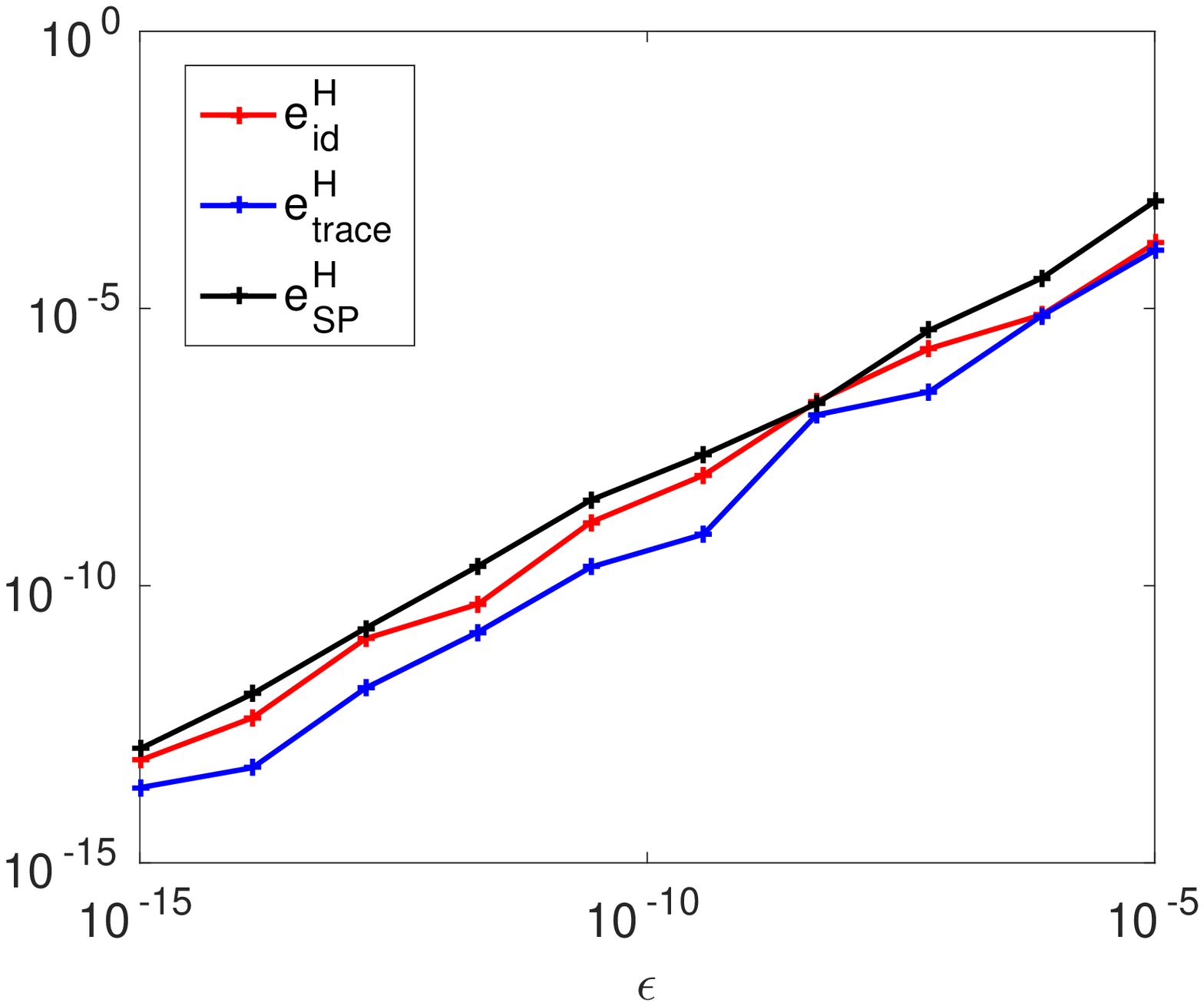}
\end{center}
\caption{Left (Example~\ref{ex:band_error_vs_gap}): Comparison of accuracy for hQDWH and QDWH applied to banded matrices with bandwidth $8$. Right (Example~\ref{ex:band_error_vs_tol}): 
Accuracy of hQDWH for different truncation tolerances.}
\label{fig:error_vs_gap_tol_band}
\end{figure}

\begin{example}[\bfseries{Accuracy versus $\epsilon$}]
\label{ex:band_error_vs_tol} 
\textnormal{We investigate the influence of the truncation tolerance $\epsilon$ on accuracy for an $10\,000\times 10\,000$ banded matrix with bandwidth $b = 8$ and 
the eigenvalues contained in $[-1, \hskip 3pt -10^{-4}] \cup [10^{-4}, \hskip 3pt 1]$. The right plot of Figure~\ref{fig:error_vs_gap_tol_band} reconfirms the observations from Example~\ref{ex:error_vs_tolerance_tridiag}. }      
\end{example}

\begin{example}[\bfseries{Breakeven point relative to {\tt eig}}]
{\rm Table~\ref{table:break_even_point_banded} shows when hQDWH becomes faster than \texttt{eig} for $n\times n$ banded matrices with eigenvalues 
contained in $[-1, \hskip 3pt -\gap] \cup [\gap, \hskip 3pt 1]$ for $\gap = 10^{-1}$ and $10^{-4}$. Compared to Table~\ref{table:break_even_point_tridiag}, the breakeven point is lower for bandwidths $b = 2$ and $b= 4$ than for bandwidth $1$. This is because {\tt eig} needs to perform tridiagonal reduction when $b\ge 2$. } 

\begin{table}[ht!]

\centering
{\renewcommand{\arraystretch}{1.2}
\begin{tabular}{c||c|c|c|c}
\backslashbox[5mm]{gap}{b} &$2$ &$4$ &$8$ &$16$\\
 \hline
 \hline
 $10^{-1}$ &$n=1250$ &$n=1750$ &$n=2500$ &$n=5250$\\
 \hline
 $10^{-4}$ &$n=1750$ &$n=2500$ &$n=5000$ &$n=9500$\\
\end{tabular}
 }
\caption{Breakeven point of hQDWH relative to \texttt{eig} applied for banded matrices with various bandwidths and spectral gaps.}
\label{table:break_even_point_banded}
\end{table}
\end{example}

\begin{example}[\bfseries{Performance versus $n$}]
\label{ex:scaling_banded_wrt_n}
\textnormal{We consider banded matrices with bandwidth $4$ and with eigenvalues contained in $[-1, \hskip 3pt -10^{-1}] \cup [10^{-1}, \hskip 3pt 1]$.
As in Example~\ref{ex:scaling_tridiag}, Figure~\ref{fig:gap_1_band_4_res} confirms that the computational time of hQDWH scales like $\mathcal{O}(n\log^2n)$ while memory scales like $\mathcal{O}(n\log n)$. Note that the maximal rank in the off-diagonal blocks
is $66$ for $n = 1\,000$, and $71$ for $n = 500\,000$.}

\begin{figure}[!h]
\begin{center}
\includegraphics[width=0.48\textwidth]{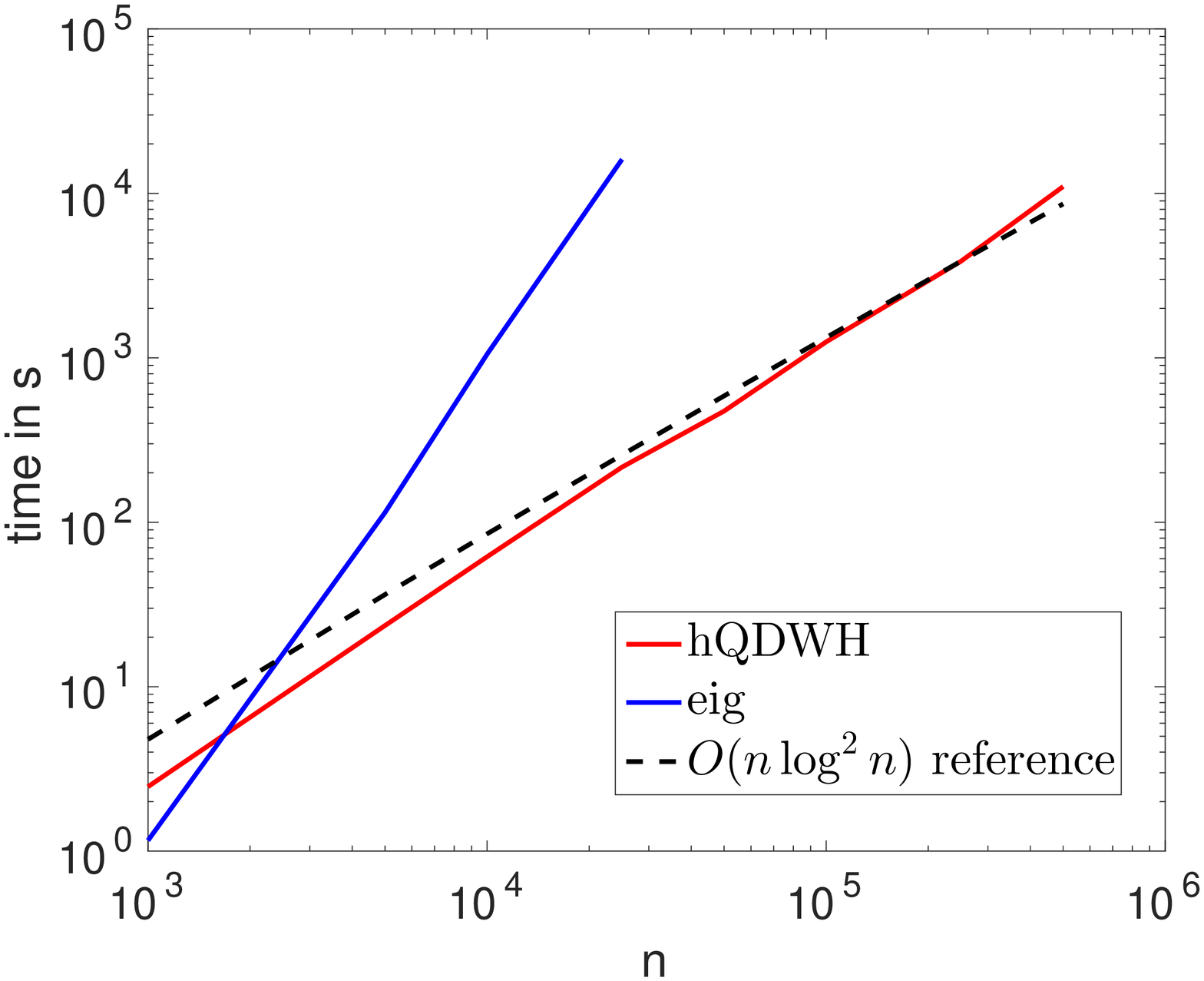}
\includegraphics[width=0.48\textwidth]{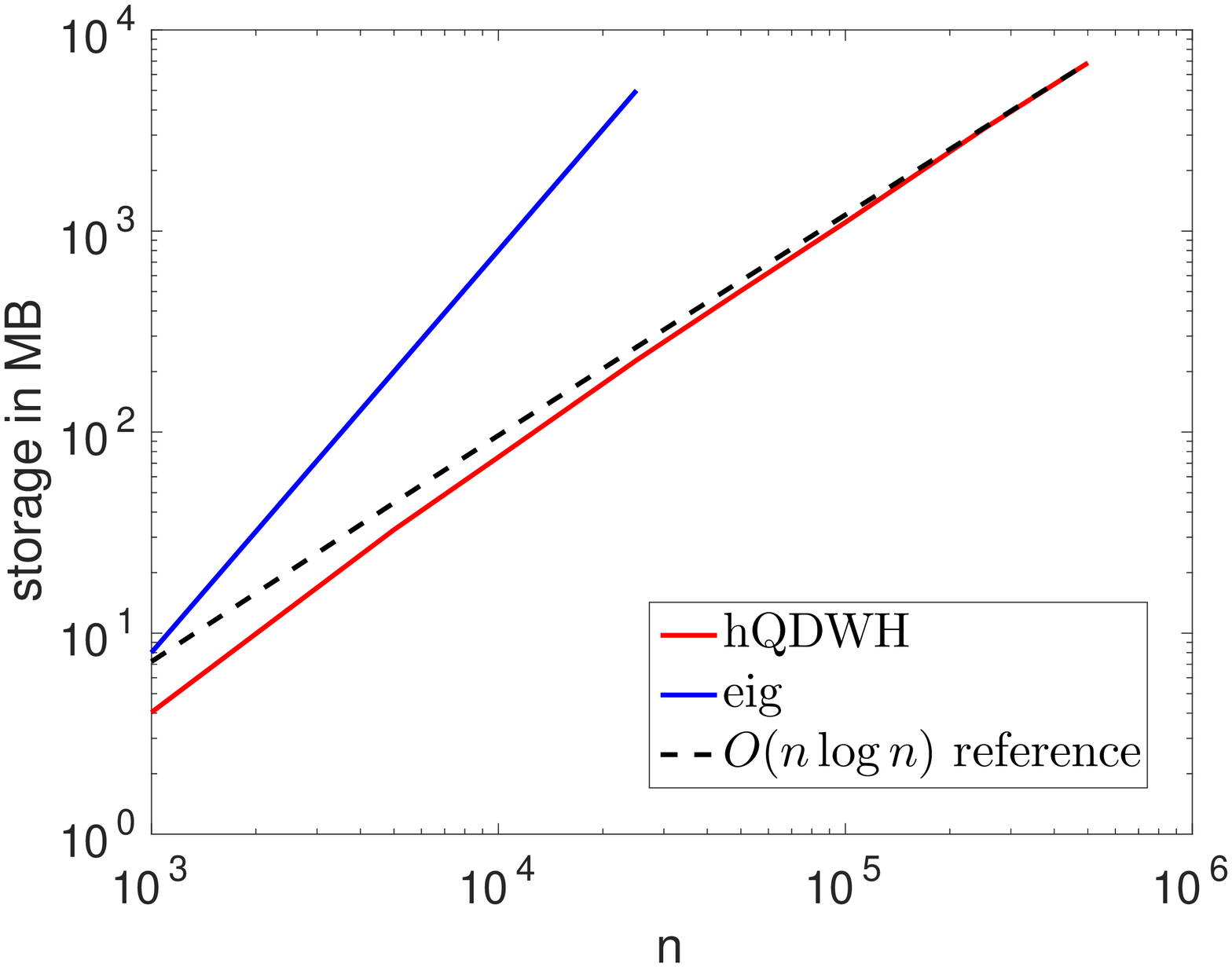}
\end{center}
\caption{Example~\ref{ex:scaling_banded_wrt_n}. Performance with respect to $n$ of hQDWH and \texttt{eig} applied to banded matrices with bandwidth $4$. Left: Computational time. Right: Memory requirements.}

\label{fig:gap_1_band_4_res}
\end{figure}
\end{example} 

\begin{example}[\bfseries{Performance versus $b$}]
\label{ex:scaling_banded_wrt_b}
\textnormal{To verify the influence of the matrix bandwidth on the performance of our algorithm, we consider $100\,000\times 100\,000$ banded matrices with 
eigenvalues contained in $[-1, \hskip 3pt -10^{-6}] \cup [10^{-6}, \hskip 3pt 1]$. Figure~\ref{fig:gap_6_band_res} clearly demonstrates that 
computational time grows quadratically while memory grows linearly with respect to the bandwidth $b$.}

\begin{figure}[!h]
\begin{center}
\includegraphics[width=0.48\textwidth]{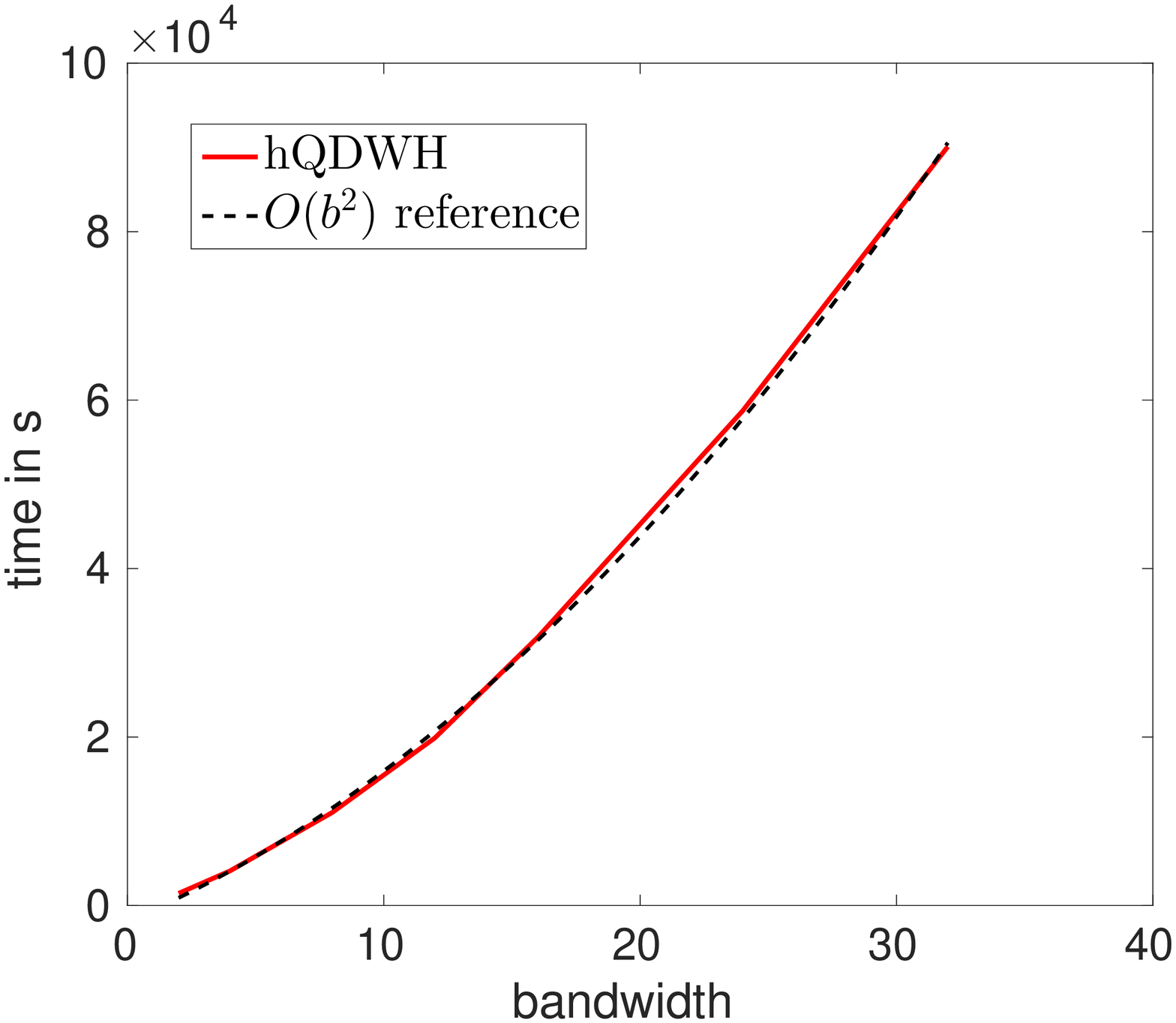}
\includegraphics[width=0.48\textwidth]{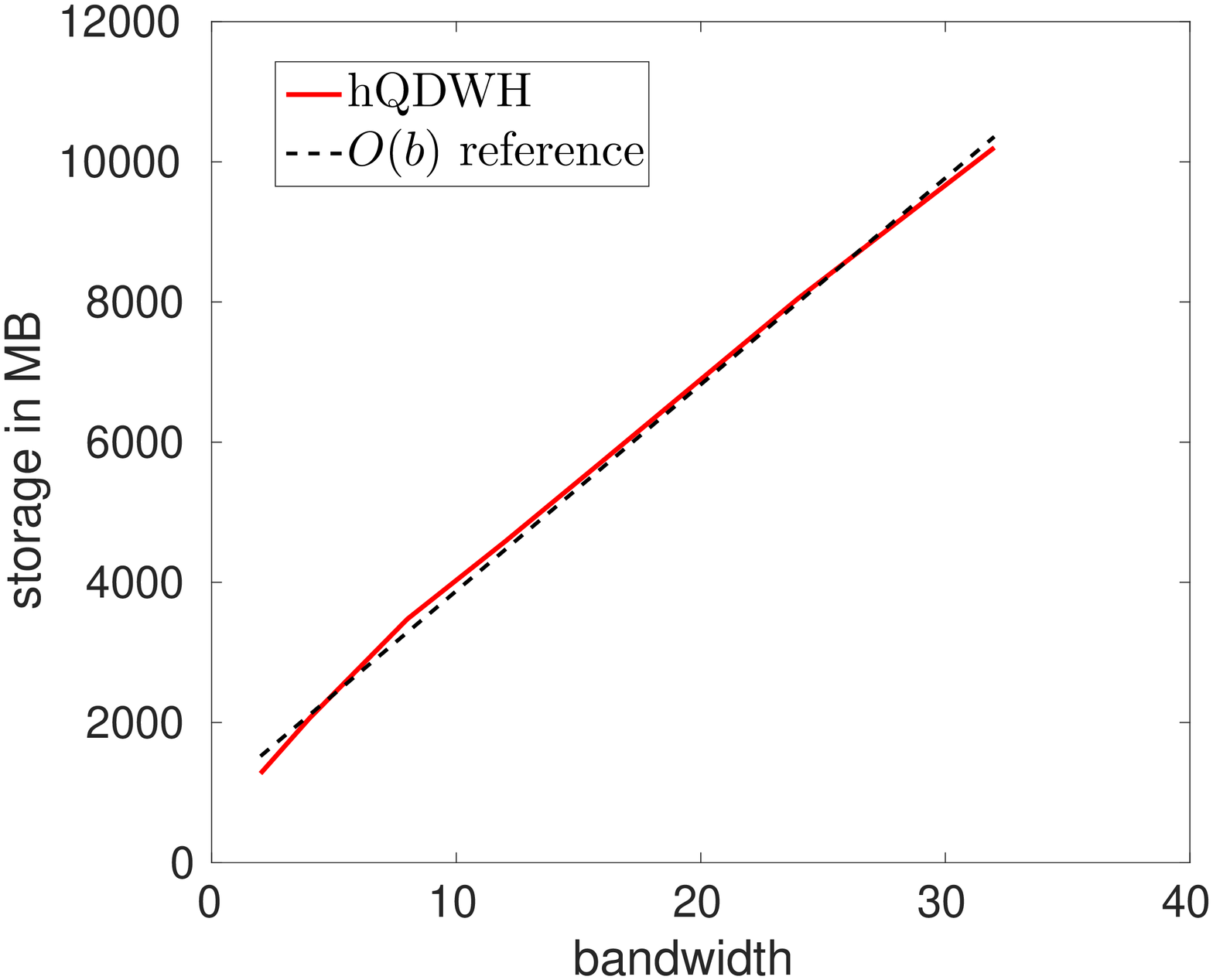}
\end{center}
\caption{Example~\ref{ex:scaling_banded_wrt_b}.
Performance with respect to bandwidth $b$ of hQDWH applied to $100\,000\times 100\,000$ banded matrices. Left: Computational time. Right: Memory requirements.}
\label{fig:gap_6_band_res}
\end{figure}
\end{example}

\section{Conclusion}

In this paper we have developed a fast algorithm for computing spectral projectors of large-scale symmetric banded matrices. For this purpose, we have tailored the 
ingredients of the QDWH algorithm, such that the overall algorithm has linear-polylogarithmic complexity. This allows us to compute highly accurate approximations to the spectral projector for very large sizes (up to $n = 1\,000\,000$ on a desktop computer) even when the involved spectral gap becomes small.

The choice of hierarchical low-rank matrix format is critical to the performance of our algorithm. Somewhat surprisingly, 
we have observed that the relatively simple HODLR format outperforms a more general $\h$-matrix format. We have not investigated the choice of a format with nested low-rank factors, 
such as HSS matrices. While such a nested format likely lowers asymptotic complexity, it presumably only pays off for larger values of $n$.

\begin{paragraph}{Acknowledgements.}
 We are grateful to Jonas Ballani, Petar Sirkovi\'c, and Michael Steinlechner 
 for helpful discussions on this paper, as well as to Stefan G{\"u}ttel for providing us insights into the approximation error results used in Section~\ref{sec:rational_approx}.
\end{paragraph}

\bibliographystyle{plain}
\bibliography{biblio}

\end{document}